\documentclass[mathpazo]{cicp}

\usepackage{color}
\usepackage{amsmath}
\usepackage{amsfonts}
\usepackage{graphicx}
\usepackage{subfigure}
\usepackage{nicefrac}
\newcommand{\Dh}{\Delta_h}
\newcommand{\nabh}{\nabla_{\! h}}

\newcommand{\hf}{\nicefrac{1}{2}}
\newcommand{\nrm}[1]{\left\| #1 \right\|}
\newcommand{\ciptwo}[2]{\left\langle #1 , #2 \right\rangle_\Omega}
\newcommand{\cipgen}[3]{\left\langle #1 , #2 \right\rangle_{#3}}

\newcommand{\nn}{\nonumber}

\newcommand\dt {{\Delta t}}

\newcommand{\eipx}[2]{\left[ #1 , #2 \right]_{\rm x}}
\newcommand{\eipy}[2]{\left[ #1 , #2 \right]_{\rm y}}

\newcommand{\eipvec}[2]{\left[ #1 , #2 \right]_{\Omega}}
\newcommand{\sumij}{\sum_{i,j=1}^N}

\begin{document}
\title{A positivity-preserving second-order BDF scheme for the Cahn-Hilliard equation with variable interfacial parameters}


 \author[Lixiu Dong et.~al.]{Lixiu Dong\affil{1}, Cheng Wang\affil{3},
       Hui Zhang\affil{2}\comma\corrauth, and Zhengru Zhang\affil{2}}
 \address{\affilnum{1}\ School of Mathematical Sciences,
          Beijing Normal University,
          Beijing 100875, P.R. China.\\
          \affilnum{2}\ Laboratory of Mathematics and Complex Systems, Ministry of Education and School of Mathematical Sciences, Beijing Normal University, Beijing 100875, P.R. China. \\
          \affilnum{3}\ Mathematics Department,
          University of Massachusetts Dartmouth,
          North Dartmouth, MA 02747, U.S.A.}
 \emails{{\tt lxdong@mail.bnu.edu.cn} (L.X. Dong),
 {\tt cwang1@umassd.edu} (C. Wang), {\tt hzhang@bnu.edu.cn} (H. Zhang),
          {\tt zrzhang@bnu.edu.cn} (Z.R. Zhang)}

\begin{abstract}
 We present and analyze a new second-order finite difference scheme for the
 Macromolecular Microsphere Composite hydrogel, Time-Dependent Ginzburg-Landau (MMC-TDGL)
 equation, a Cahn-Hilliard equation with Flory-Huggins-deGennes energy potential.
 This numerical scheme with unconditional energy stability is based on the Backward
 Differentiation Formula (BDF) method in time derivation combining with Douglas-Dupont
 regularization term. In addition, we present a point-wise bound of the numerical solution
  for the proposed scheme in the theoretical level. For the convergent analysis, we treat
 three nonlinear logarithmic terms as a whole and deal with all logarithmic terms directly
 by using the property that the nonlinear error inner product is always non-negative.
 Moreover, we present the detailed convergent analysis in $\ell^\infty (0,T; H_h^{-1}) \cap \ell^2 (0,T; H_h^1)$ norm.
 At last, we use the local Newton approximation and multigrid method to solve the nonlinear
 numerical scheme, and various numerical results are presented, including the numerical
 convergence test, positivity-preserving property test, spinodal decomposition, energy dissipation
  and mass conservation properties.
\end{abstract}


\maketitle

\section{Introduction}
\label{sec1}
The Time-Dependent Ginzburg-Landau mesoscopic model for the macromolecular microsphere composite(MMC) hydrogel, called MMC-TDGL equation, was recently proposed in \cite{Zhai2012Investigation} as a new approach to simulating a reticular structure and phase transition process of MMC hydrogel. The MMC-TDGL model accounts for the periodic network structure of MMC through a coarse-grained free energy functional of the Flory-Huggins-deGennes type \cite{Zhai2012Investigation}. The model can describe the growth detail of the well-defined structures intermittent phenomenon with increasing reaction temperature, and many other chemical observable phenomena \cite{Huang2010A}. The idea is that a conserved field variable represents the concentration of one of the components of the mixture(or sometimes, the difference between the concentration of the two components of a binary mixture). The approach is derived via Boltzmann entropy theorem.

Allen-Cahn and Cahn-Hilliard equations are the prototypical models for gradient flows with the Ginzburg-Landau or Flory-Huggins free energy. In some cases, certain stochastic force term has been added in the model, such as the Cahn-Hilliard-Cook model. This model can simulate the structural evolution of mixtures with polymers and block copolymers \cite{FPJ1953} and the phase separation of the small molecule systems including binary alloys, fluid mixtures, inorganic glasses \cite{Chakrabarti1990Dynamics}.
Concerning the computation and analysis of these models, Du et al had a series of works \cite{DD1,DD2,DFY}. Yang et al presented an invariant energy quadratization (IEQ)  approximation \cite{ZWY,ZYW,SY11,SY12,Yang2019On}. Chen et al used the phase field method to investigate composite materials and presented some numerical methods\cite{CKDC,DCZ}. Shen et al designed a few high-order energy stability preserving numerical schemes and provided the corresponding error estimates \cite{SY1,SY2,SY3,SY5,SXY,SX}. These works investigated the nucleation by using string method in virtue of stochastic Allen-Cahn and Cahn-Hilliard equations \cite{zlz}. For the MMC-TDGL equation, Li et al \cite{Li2016An} also have performed some numerical simulations. Also see the related works~\cite{Bosch2016Preconditioning, Diegel2019}, etc.

The convex splitting approach advanced by Eyre \cite{eyre98} is one of the popular energy stable methods. The idea is that the energy admits a splitting into purely convex and concave parts, that is, $E=E_c-E_e,$  where $E_c$ and $E_e$ are both convex. Recently, such an idea has also been applied to a wide class of gradient flows, including either first or second order accurate in both time and space. See the related works for the PFC equation and the modified PFC (MPFC) equation~\cite{Dong2016Convergence,wang10c,wang11a,Wise2009AN}; the epitaxial thin film growth models~\cite{chen12,chen14,cheng2018c,shen12,wang10a,Feng2017A}, and the Cahn-Hilliard flow coupled with fluid motion~\cite{CLW,chen17b,diegel15a,diegel17,han15,liuY17,wise10}, etc.
One well-known drawback of the first order convex splitting approach is that an extra dissipation added to ensure unconditionally stability also introduces a significant amount of numerical error~\cite{Christlieb2014High}. Due to this, second-order energy stable methods have been highly desirable. Recently, a second order convex splitting scheme based on the Crank-Nicolson temporal approximation for solving the MMC-TDGL equation has been proposed in~\cite{Xiao2017A}, however, its convergence analysis is still a large challenge.

The goal of this paper is to extend the convex-splitting framework to develop a second order in both time and space for the MMC-TDGL equation. While some of the technique that worked for the Cahn-Hilliard schemes with the Flory-Huggins type potential are appropriate, the analysis for the positivity-preserving and convergence are more difficulty for the MMC-TDGL scheme mainly owing to the variable diffusive coefficient, called the deGennes diffusion coefficient.

In this paper, we design an unconditionally stable, unconditionally unique solvable, second order in time and space, and convergent scheme for the MMC-TDGL equation based on the convex-splitting method. The scheme is based on the 2nd BFD temporal approximation and the centered difference method in space for the MMC-TDGL equation. In more details, the derivative function with respect to time is approximated by the BDF 3-point stencil. Based on the idea of convex splitting, we treat the convex part implicitly and the linear part explicitly using the second-order Adams-Bashforth extrapolation formula. By a careful calculation, it is hard to get the energy stability owing to the explicit expression of the linear part. To overcome this difficulty for the original numerical scheme, we adopt the similar technique in~\cite{Feng2017A,cheng18ch,CHBDF2}, adding a second order Douglas-Dupont regularization of the form $A\dt \Delta_h (\phi^{n+1}-\phi^n)$. The resulting scheme holds the modified discrete energy non-increasing under a restriction $ A>\chi^2 \rho^2$, which would be accepted as the numerical scheme is a three-level scheme.

In addition, in the continuous case the phase variable remains in the interval of $(0,\nicefrac{1}{\rho})$. In the discrete case, the proposed numerical scheme still keeping this property is highly desired. In the earlier work~\cite{elliott92a}, the author analyzed a fully discrete finite element scheme based on the backward Euler approximation for the Cahn-Hilliard equation with a logarithmic free energy and obtained some theoretical results about the existence, uniqueness and the positivity property of the numerical solution, but this scheme is not unconditionally energy stable.

In the recent literature \cite{Chen2017A}, the authors presented a discrete finite difference numerical scheme based on the convex splitting method of the free energy with logarithmic potential, and established a theoretical justification of the positivity property, regardless of time step size. In this paper, we will adopt similar techniques  in~\cite{Chen2017A} to estimate the positivity property and the convergence analysis, respectively. For the positivity property, in details, the fully discrete numerical scheme is equivalent to a minimization of a strictly convex discrete energy functional, so we can transform the positivity preserving problem of the numerical solution into the problem that the minimizer of this functional could not occur on the boundary points. Due to the logarithmic terms implicitly, we can make use of the following subtle fact: the singular nature of the logarithmic function guarantees that such a minimizer could not occur on a boundary point at all. Although the extra term $A\dt \Delta_h (\phi^{n+1}-\phi^n)$ is added into the numerical scheme, it does not matter because the logarithmic function changes faster than the linear function as the phase variable approaches the boundary points. It is obvious that if the logarithmic term is explicit, such an estimate could not be derived by this method.     Moreover, the term associated with the deGennes coefficient is very challenging. With the help of the following inequality {\em  $\frac{1}{2} \kappa'(\phi_1) (\phi_2-\phi_1)  \le \kappa(\phi_2), \forall \phi_1,\phi_2 \in (0,1) $} about the deGennes coefficient, which plays an essential role in the analysis of the positivity preserving, we can establish the positivity-preserving property. Also see~\cite{Dong2019a, Lv2018} for related discussions.

The key difficulty in the convergence analysis is associated with the logarithmic potential term. In general, when the nonlinear term is a polynomial approximation, the bound estimate of maximum norm of the numerical solution is necessary to justify the convergence analysis \cite{Dong2016Convergence}, but it is not sufficient to solve the case with the logarithmic potential.
In this paper, for the error estimate, we take inner product with error at the time step $t^{n+1}$. We treat the three nonlinear logarithmic terms as a whole in a rough way and then make full use of the convexity of energy about these nonlinear terms to directly deal with all logarithmic terms, because the convexity of energy indicates the corresponding nonlinear error inner product is always non-negative. Moreover, it is observation that in the chemical potential the surface diffusion term with concentration-dependent deGennes type coefficient can be decomposed into two convex terms: one term depended on its convexity can be analyzed in a manner similar to the logarithmic term, and the other term can be used to control the explicit error estimate associated with the linear expansive term. In order to deal with the temporal derivation approximation term, we also introduce a weighted norm. In turn, the convergence analysis could go through for the proposed scheme.

The rest of the paper is organized as follows. In Section 2, we present the MMC-TDGL equation. In Section 3, we present the 2nd BDF numerical scheme, and the positivity-preserving property of the numerical solution is provided in Section 4. The theoretical analysis of the modified energy stability is estimated in Section 5. The detailed convergence analysis is given by Section 6. Some numerical results are presented in Section 7. Concluding remarks are made in Section 8.

 \section{The model equation: MMC-TDGL equation}\label{sec:model}
 We consider a bounded domain $\Omega \subset \mathbb{R}^2$.
 For any $\phi \in H^1 (\Omega)$, with a point-wise bound, $\phi \in (0, \nicefrac{1}{\rho})\subset (0,1)$, the energy functional is the form of
     \begin{equation}
	\label{CH energy}
E(\phi)=\int_{\Omega}\left( S(\phi) + H(\phi) +\kappa(\phi)|\nabla \phi|^2\right) d {\bf x} ,
	 \end{equation}
where $S(\phi) + H(\phi)$ is the reticular free energy density for the MMC hydrogels
     \begin{equation}
S(\phi)= \frac{\phi}{\tau} \ln \frac{\alpha \phi}{\tau} + \frac{\phi}{N_1} \ln \frac{\beta \phi}{\tau} + (1- \rho\phi) \ln (1- \rho\phi),\quad H(\phi)= \chi\phi(1-\rho\phi),
    \label{CH_energy_SH}
	\end{equation}
and $\kappa(\phi)$ is the deGennes coefficient
    \begin{equation}
\kappa(\phi)= \frac{1}{36\phi(1-\phi)}.
    \label{CH_energy_K}
    \end{equation}
In this model, we denote by $\chi$ the Huggins interaction parameter, by $N_1$ the degree of polymerization of the polymer chains, and by $N_2$, which does not appear explicitly in \eqref{CH energy}, the relative volume of one macromolecular microsphere. The other numbers $\alpha,\beta,\tau$ and $\rho$ depend on $N_2$ and $N_1$, as given by
     \[
\alpha= \pi\left(\sqrt{\frac{N_2}{\pi}}+ \frac{N_1}{2}\right)^2,\quad \beta= \frac{\alpha}{\sqrt{\pi N_2}},\quad \tau= \sqrt{\pi N_2}N_1, \quad \rho= 1+ \frac{N_2}{\tau}.
     \]
Note that all these parameters are positive. Besides, $\rho$ is a little greater than one. The modeling detail can be referred to \cite{Zhai2012Investigation}.

In turn, the MMC-TDGL equation for the MMC hydrogels becomes the following $H^{-1}$ gradient flow associated with the given energy functional \eqref{CH energy}:
     \begin{align}
\partial_t\phi = \Delta \mu ,  \quad \mu := \delta_\phi E =& S'(\phi) + H' (\phi) + \kappa'(\phi) | \nabla \phi |^2 - 2 \nabla \cdot(\kappa (\phi) \nabla \phi) \label{CH equation-0}
\\
  =&
   (\frac{1}{\tau}+\frac{1}{N_1}) \ln \phi - \rho \ln(1-\rho \phi) -2\chi\rho\phi  \nonumber
\\
&+ \frac{2\phi-1}{36\phi^2(1-\phi)^2} | \nabla \phi |^2 - \nabla \cdot \left(\frac{\nabla \phi}{18 \phi(1-\phi)}\right) .  \nonumber
	 \end{align}
Here we have discarded the constant terms in the representation for the chemical potential $\mu$, since these terms will not play any role in the $H^{-1}$ gradient flow.
\section{The numerical scheme}\label{sec:numerical scheme}
In the spatial discretization, the centeral difference approximation is applied. We recall some basic notations of this methodology.
\subsection{Discretization of space and a few preliminary estimates}\label{subsec:finite difference}
We use the notations and results for some discrete functions and
operators from~\cite{Wise2009AN,wise10,guo16}.
Let $\Omega = (0,L_x)\times(0,L_y)$, for simplicity, we assume $L_x =L_y =: L > 0$. Let $N\in\mathbb{N}$ be given, and define the grid spacing $h := \nicefrac{L}{N}$.  We also assume -- but only for simplicity of notation, ultimately -- that the mesh spacing in the $x$ and $y$-directions are the same. The following two uniform, infinite grids with grid spacing $h>0$, are introduced
	\[
E := \{ p_{i+\hf} \ |\ i\in {\mathbb{Z}}\}, \quad C := \{ p_i \ |\ i\in {\mathbb{Z}}\},
	\]
where $p_i = p(i) := (i-\hf)\cdot h$. Consider the following 2-D discrete $N^2$-periodic function spaces:
    \begin{eqnarray*}
	\begin{aligned}
{\mathcal C}_{\rm per} &:= \left\{\nu: C\times C
\rightarrow {\mathbb{R}}\ \middle| \ \nu_{i,j} = \nu_{i+\alpha N,j+\beta N}, \ \forall \, i,j,\alpha,\beta\in \mathbb{Z} \right\},
	\\
{\mathcal E}^{\rm x}_{\rm per} &:=\left\{\nu: E\times C\rightarrow {\mathbb{R}}\ \middle| \ \nu_{i+\frac12,j}= \nu_{i+\frac12+\alpha N,j+\beta N}, \ \forall \, i,j,\alpha,\beta\in \mathbb{Z}\right\} .
	\end{aligned}
	\end{eqnarray*}
Here we are using the identification $\nu_{i,j} = \nu(p_i,p_j)$, \emph{et cetera}. The space  ${\mathcal E}^{\rm y}_{\rm per}$ is analogously defined. The function of ${\mathcal C}_{\rm per}$ is called {\emph{cell-centered function}}. The function of ${\mathcal E}^{\rm x}_{\rm per}$ and ${\mathcal E}^{\rm y}_{\rm per}$,  is called {\emph{edge-centered function}}.  We also define the mean zero space
    \[
\mathring{\mathcal C}_{\rm per}:=\left\{\nu\in {\mathcal C}_{\rm per} \ \middle| 0 = \overline{\nu} :=  \frac{h^2}{| \Omega|} \sum_{i,j=1}^N \nu_{i,j} \right\} .
	\]
In addition, $\vec{\mathcal{E}}_{\rm per}$ is defined as $\vec{\mathcal{E}}_{\rm per} := {\mathcal E}^{\rm x}_{\rm per}\times {\mathcal E}^{\rm y}_{\rm per}$.
We now introduce the difference and average operators on the spaces:
	\begin{eqnarray*}
&& A_x \nu_{i+\hf,j} := \frac{1}{2}\left(\nu_{i+1,j} + \nu_{i,j} \right), \quad D_x \nu_{i+\hf,j} := \frac{1}{h}\left(\nu_{i+1,j} - \nu_{i,j} \right),
	\\
&& A_y \nu_{i,j+\hf} := \frac{1}{2}\left(\nu_{i,j+1} + \nu_{i,j} \right), \quad D_y \nu_{i,j+\hf} := \frac{1}{h}\left(\nu_{i,j+1} - \nu_{i,j} \right) ,
	\end{eqnarray*}
with $A_x,\, D_x: {\mathcal C}_{\rm per}\rightarrow{\mathcal E}_{\rm per}^{\rm x}$, $A_y,\, D_y: {\mathcal C}_{\rm per}\rightarrow{\mathcal E}_{\rm per}^{\rm y}$.
Likewise,
    \begin{eqnarray*}
&& a_x \nu_{i, j} := \frac{1}{2}\left(\nu_{i+\hf, j} + \nu_{i-\hf, j} \right),	 \quad d_x \nu_{i, j} := \frac{1}{h}\left(\nu_{i+\hf, j} - \nu_{i-\hf, j} \right),
	\\
&& a_y \nu_{i,j} := \frac{1}{2}\left(\nu_{i,j+\hf} + \nu_{i,j-\hf} \right),	 \quad d_y \nu_{i,j} := \frac{1}{h}\left(\nu_{i,j+\hf} - \nu_{i,j-\hf} \right),
	\end{eqnarray*}
with $a_x,\, d_x : {\mathcal E}_{\rm per}^{\rm x}\rightarrow{\mathcal C}_{\rm per}$, $a_y,\, d_y : {\mathcal E}_{\rm per}^{\rm y}\rightarrow{\mathcal C}_{\rm per}$.
The discrete gradient operator $\nabh:{\mathcal C}_{\rm per}\rightarrow \vec{\mathcal{E}}_{\rm per}$ is given by
    \[
\nabh\nu_{i,j} =\left( D_x\nu_{i+\hf, j},  D_y\nu_{i, j+\hf}\right) ,
	\]
and the discrete divergence $\nabh\cdot :\vec{\mathcal{E}}_{\rm per} \rightarrow {\mathcal C}_{\rm per}$ is defined via
	\[
\nabh\cdot\vec{f}_{i,j} = d_x f^x_{i,j}	+ d_y f^y_{i,j},
	\]
where $\vec{f} = (f^x,f^y)\in \vec{\mathcal{E}}_{\rm per}$. The standard 2-D discrete Laplacian, $\Delta_h : {\mathcal C}_{\rm per}\rightarrow{\mathcal C}_{\rm per}$, becomes
    \begin{align*}
\Delta_h \nu_{i,j} := &  d_x(D_x \nu)_{i,j} + d_y(D_y \nu)_{i,j}
	\\
= & \ \frac{1}{h^2}\left( \nu_{i+1,j}+\nu_{i-1,j}+\nu_{i,j+1}+\nu_{i,j-1} - 4\nu_{i,j}\right).
	\end{align*}
More generally, suppose $\mathcal{D}$ is a periodic \emph{scalar} function that is defined at all of the edge-center points and $\vec{f}\in\vec{\mathcal{E}}_{\rm per}$, then $\mathcal{D}\vec{f}\in\vec{\mathcal{E}}_{\rm per}$, assuming point-wise multiplication, and we may define
	\[
\nabla_h\cdot \big(\mathcal{D} \vec{f} \big)_{i,j} = d_x\left(\mathcal{D}f^x\right)_{i,j}  + d_y\left(\mathcal{D}f^y\right)_{i,j}  .
	\]
Specifically, if $\nu\in \mathcal{C}_{\rm per}$, then $\nabla_h \cdot\left(\mathcal{D} \nabla_h  \ \ \right):\mathcal{C}_{\rm per} \rightarrow \mathcal{C}_{\rm per}$ is defined point-wise via
	\[
\nabla_h\cdot \big(\mathcal{D} \nabla_h \nu \big)_{i,j} = d_x\left(\mathcal{D}D_x\nu\right)_{i,j}  + d_y\left(\mathcal{D} D_y\nu\right)_{i,j} .
	\]
Now we are ready to define the following grid inner products:
	\begin{equation*}
	\begin{aligned}
\ciptwo{\nu}{\xi} &:= h^2\sum_{i,j=1}^N  \nu_{i,j}\, \xi_{i,j},\ \nu,\, \xi\in {\mathcal C}_{\rm per},\,
&\eipx{\nu}{\xi} := \ciptwo{a_x(\nu\xi)}{1} ,\ \nu,\, \xi\in{\mathcal E}^{\rm x}_{\rm per},
\\
\eipy{\nu}{\xi} &:= \ciptwo{a_y(\nu\xi)}{1} ,\ \nu,\, \xi\in{\mathcal E}^{\rm y}_{\rm per},
	\end{aligned}
	\end{equation*}
	\[
\eipvec{\vec{f}_1}{\vec{f}_2} : = \eipx{f_1^x}{f_2^x}	+ \eipy{f_1^y}{f_2^y} , \quad \vec{f}_i = (f_i^x,f_i^y) \in \vec{\mathcal{E}}_{\rm per}, \ i = 1,2.
	\]
In turn, the following norms could be appropriately introduced for cell-centered functions. If $\nu\in {\mathcal C}_{\rm per}$, then $\nrm{\nu}_2^2 := \ciptwo{\nu}{\nu}$; $\nrm{\nu}_p^p := \ciptwo{|\nu|^p}{1}$, for $1\le p< \infty$, and $\nrm{\nu}_\infty := \max_{1\le i,j\le N}\left|\nu_{i,j}\right|$.
We define  norms of the gradient as follows: for $\nu\in{\mathcal C}_{\rm per}$,
	\[
\nrm{ \nabla_h \nu}_2^2 : = \eipvec{\nabh \nu }{ \nabh \nu } = \eipx{D_x\nu}{D_x\nu} + \eipy{D_y\nu}{D_y\nu},
	\]
and, more generally, for $1\le p<\infty$,
	\begin{equation*}
\nrm{\nabla_h \nu}_p := \left( \eipx{|D_x\nu|^p}{1} + \eipy{|D_y\nu|^p}{1} \right)^{\frac1p} .
	\end{equation*}
Higher order norms can be similarly formulated. For example,
	\[
\nrm{\nu}_{H_h^1}^2 : =  \nrm{\nu}_2^2+ \nrm{ \nabla_h \nu}_2^2, \quad \nrm{\nu}_{H_h^2}^2 : =  \nrm{\nu}_{H_h^1}^2  + \nrm{ \Delta_h \nu}_2^2.
	\]
\begin{lemma}
	\label{lemma1}
For any $\psi, \nu \in {\mathcal C}_{\rm per}$ and any $\vec{f}\in\vec{\mathcal{E}}_{\rm per}$, the following summation by parts formulas are valid:
	\begin{equation}
\ciptwo{\psi}{\nabla_h\cdot\vec{f}} = - \eipvec{\nabla_h \psi}{ \vec{f}}, \quad \ciptwo{\psi}{\nabla_h\cdot \left(\mathcal{D}\nabla_h\nu\right)} = - \eipvec{\nabla_h \psi }{ \mathcal{D}\nabla_h\nu} .
\label{lemma 1-0}
	\end{equation}
	\end{lemma}
To facilitate the convergence analysis, we need to introduce a discrete analogue of the space $H_{per}^{-1}\left(\Omega\right)$, as outlined in~\cite{wang11a}. Suppose that $\mathcal{D}$ is a positive, periodic scalar function defined at all of edge-center points. For any $\phi\in{\mathcal C}_{\rm per}$, there exists a unique $\psi\in\mathring{\mathcal C}_{\rm per}$ that solves
	\begin{eqnarray*}
\mathcal{L}_{\mathcal{D}}(\psi):= - \nabla_h \cdot\left(\mathcal{D}\nabla_h \psi\right) = \phi - \overline{\phi} ,
	\end{eqnarray*}
where $\overline{\phi} := |\Omega|^{-1}\ciptwo{\phi}{1}$. We equip this space with a bilinear form: for any $\phi_1,\, \phi_2\in \mathring{\mathcal C}_{\rm per}$, define
    \begin{equation*}
\cipgen{ \phi_1 }{ \phi_2 }{\mathcal{L}_{\mathcal{D}}^{-1}} := \eipvec{\mathcal{D}\nabla_h \psi_1 }{ \nabla_h \psi_2 },
	\end{equation*}
where $\psi_i\in\mathring{\mathcal C}_{\rm per}$ is the unique solution to
	\begin{equation*}
\mathcal{L}_{\mathcal{D}}(\psi_i):= - \nabla_h \cdot\left(\mathcal{D}\nabla_h \psi_i\right)  = \phi_i, \quad i = 1, 2.
	\end{equation*}
The following identity~\cite{wang11a} is easy to prove via summation-by-parts:
	\begin{equation}
\cipgen{\phi_1 }{ \phi_2 }{\mathcal{L}_{\mathcal{D}}^{-1}} = \ciptwo{\phi_1}{ \mathcal{L}_{\mathcal{D}}^{-1} (\phi_2) } = \ciptwo{ \mathcal{L}_{\mathcal{D}}^{-1} (\phi_1) }{\phi_2 },
	\end{equation}
and since $\mathcal{L}_{\mathcal{D}}$ is symmetric positive definite, $\cipgen{ \ \cdot \ }{\ \cdot \ }{\mathcal{L}_{\mathcal{D}}^{-1}}$ is an inner product on $\mathring{\mathcal C}_{\rm per}$~\cite{wang11a}. When $\mathcal{D}\equiv 1$, we drop the subscript and write $\mathcal{L}_{1} = \mathcal{L}$, and in this case we usually write $\cipgen{ \ \cdot \ }{\ \cdot \ }{\mathcal{L}_{\mathcal{D}}^{-1}} =: \cipgen{ \ \cdot \ }{\ \cdot \ }{-1,h}$. In the gerneral setting, the norm associated to this inner product is denoted $\nrm{\phi}_{\mathcal{L}_{\mathcal{D}}^{-1}} := \sqrt{\cipgen{\phi }{ \phi }{\mathcal{L}_{\mathcal{D}}^{-1}}}$, for all $\phi \in \mathring{\mathcal C}_{\rm per}$. If $\mathcal{D}\equiv 1$, we write $\nrm{\, \cdot \, }_{\mathcal{L}_{\mathcal{D}}^{-1}} =: \nrm{\, \cdot \, }_{-1,h}$.\\
The following preliminary results is associated with the existence of a convex splitting and the error analysis in Section 6.
\begin{proposition}
\label{convex splitting}
(1) S and -H are both convex in $(0,\nicefrac{1}{\rho})$, where S and H are defined by \eqref{CH_energy_SH};\\
(2) $K(u,v):=\kappa(u) v^2$ is convex in $(0,\nicefrac{1}{\rho})\times \mathbb{R}$, where $\kappa$ is defined by \eqref{CH_energy_K};\\
(3) $K_1(u,v):=\left(\kappa(u)-\frac{1}{36}\right) v^2$  and $K_2(v):=\frac{1}{36} v^2$ are both convex in $(0,\nicefrac{1}{\rho})\times \mathbb{R}$ and $\mathbb{R}$, respectively.
\end{proposition}
\begin{proof}
(1) For $S, H, K_2$, differentiating $S, H, K_2$ twice, we obtain
     \[
S''(\phi) = \left( \frac{1}{\tau}+\frac{1}{N_1} \right) \frac{1}{\phi} + \frac{\rho^2}{1-\rho \phi},\quad H''(\phi) = -2\chi \rho, \quad K_2''(v) = \frac{1}{18}>0.
     \]
When $\phi \in (0,\nicefrac{1}{\rho})$, we have $S''(\phi)>0$ and $H''(\phi)< 0$.\\

(2) For $K(u,v):=\kappa(u) v^2$, by some careful calculations, we obtain the Hessian matrix of $K$:\\
     \begin{equation}
\nabla^2 K =
\left(
     \begin{array}{cc}
    \frac{(3u^2-3u+1)v^2}{18u^3(1-u)^3} & \frac{(2u-1)v}{18u^2(1-u)^2} \\
    \frac{(2u-1)v}{18u^2(1-u)^2} & \frac{1}{18u(1-u)} \\
     \end{array}
\right).\nonumber
     \end{equation}
The first-order principal minors of the matrix $\nabla^2 K$ are
     \[
D_1 = \frac{(3u^2-3u+1)v^2}{18u^3(1-u)^3}, \quad D_2 = \frac{1}{18u(1-u)}.
     \]
The second-order principal minor is
     \[
D_{12} = det(\nabla^2 K) = \frac{v^2}{18^2u^3(1-u)^3}.
     \]
These principal minors are all non-negative when $u \in (0,\nicefrac{1}{\rho})$ and $v \in \mathbb{R}$. The Hessian matrix $\nabla^2 K$ is positive semi-definite and thus $K$ is convex in $(0,\nicefrac{1}{\rho})\times \mathbb{R}$.\\

(3) For $K_1(u,v):=\left(\kappa(u)-\frac{1}{36}\right) v^2$, by some careful calculations, we obtain the Hessian matrix of $K_1$:\\
     \begin{equation}
\nabla^2 K_1 =
\left(
     \begin{array}{cc}
    \frac{(3u^2-3u+1)v^2}{18u^3(1-u)^3} & \frac{(2u-1)v}{18u^2(1-u)^2} \\
    \frac{(2u-1)v}{18u^2(1-u)^2} & \frac{u^2-u+1}{18u(1-u)} \\
     \end{array}
\right).\nonumber
     \end{equation}
The first-order principal minors of the matrix $\nabla^2 K_1$ are
     \[
D_1 = \frac{(3u^2-3u+1)v^2}{18u^3(1-u)^3}, \quad D_2 = \frac{u^2-u+1}{18u(1-u)}.
     \]
The second-order principal minor is
     \[
D_{12} = det(\nabla^2 K_1) = \frac{3}{18^2u^2(1-u)^2}.
     \]
These principal minors are all non-negative when $u \in (0,\nicefrac{1}{\rho})$ and $v \in \mathbb{R}$. The Hessian matrix $\nabla^2 K_1$ is positive semi-definite and thus $K_1$ is convex in $(0,\nicefrac{1}{\rho})\times \mathbb{R}$.
\end{proof}
With the preparation above, we turn to discussing the discrete energy based on a convex splitting.

Define the discrete energy $F :{\mathcal C}_{\rm per} \rightarrow \mathbb{R}$ as
     \begin{align*}
F(\phi) &= \ciptwo{S(\phi)+ H(\phi) + \kappa(\phi) (a_x((D_x\phi)^2) + a_y((D_y\phi)^2)) }{1}
    \\
    &= h^2 \sumij \left(S(\phi_{i,j}) + H(\phi_{i,j}) + \kappa(\phi_{i,j}) (a_x((D_x\phi)^2)_{i,j} + a_y((D_y\phi)^2)_{i,j} )\right).
	 \end{align*}
Define
     \[
F_S(\phi) = \ciptwo{S(\phi)}{1} = h^2 \sumij S(\phi_{i,j}),
     \]
     \[
F_e(\phi) = F_H(\phi) = \ciptwo{-H(\phi)}{1} = -h^2 \sumij H(\phi_{i,j}),
     \]
     \begin{align*}
F_{K_1}(\phi) =& \ciptwo{(\kappa(\phi)- \frac{1}{36} ) (a_x((D_x\phi)^2) + a_y((D_y\phi)^2))}{1} \\
=& h^2 \sumij (\kappa(\phi_{i,j})- \frac{1}{36} ) (a_x((D_x\phi)^2)_{i,j} + a_y((D_y\phi)^2)_{i,j} ),
     \end{align*}
     \[
F_{K_2}(\phi) = h^2 \sumij  \frac{1}{36} (a_x((D_x\phi)^2)_{i,j} + a_y((D_y\phi)^2)_{i,j} ) = \frac{1}{36} \nrm{ \nabla_h \phi}_2^2 ,
     \]
     \[
F_c(\phi) = F_S(\phi) + F_{K_1}(\phi) + F_{K_2}(\phi) .
     \]
\begin{lemma}
(Existence of a convex splitting) Assume that $\phi \in {\mathcal C}_{\rm per}$. We have
$$F(\phi) = F_c(\phi) - F_e(\phi) = F_S(\phi) + F_{K_1}(\phi) + F_{K_2}(\phi) - F_H(\phi), $$
where $F_c(\phi), F_e(\phi), F_S(\phi), F_{K_1}(\phi),  F_{K_2}(\phi)$ and $ F_H(\phi) $ are both convex.
\end{lemma}
\subsection{The fully discrete numerical scheme}

We follow the idea of the convexity splitting and consider the following semi-implicit, fully discrete scheme: for $n\geq 1$, given $\phi^n,\phi^{n-1}\in \mathcal{C}_{\rm per}$, find $\phi^{n+1}$, $\mu^{n+1}\in  \mathcal{C}_{\rm per}$, such that
\begin{small}
	\begin{align}	
\frac{3\phi^{n+1} - 4\phi^n + \phi^{n-1}}{2\dt}=& \ \Delta_h \mu^{n+1}  ,
	\label{scheme-CH_LOG-1}
    \\
\mu^{n+1} =& \ \delta_\phi F_c(\phi^{n+1}) - \delta_\phi F_e(2\phi^n-\phi^{n-1}) - A \dt \Dh(\phi^{n+1}-\phi^n) \nonumber\\
=& \ \delta_\phi F_S(\phi^{n+1}) + \delta_\phi F_{K_1}(\phi^{n+1}) + \delta_\phi F_{K_2}(\phi^{n+1})\nonumber\\
&- \delta_\phi F_H(2\phi^n-\phi^{n-1})- A \dt \Dh(\phi^{n+1}-\phi^n)\nonumber\\
=& \ S'(\phi^{n+1})+ \kappa'(\phi^{n+1})\left(a_x((D_x \phi^{n+1})^2)+a_y((D_y \phi^{n+1})^2)\right)      \label{scheme-mu-0}\\
             &- 2d_x(A_x\kappa(\phi^{n+1})D_x \phi^{n+1})
             - 2d_y(A_y\kappa(\phi^{n+1})D_y \phi^{n+1})
             	\nonumber\\
             &+ H'(2\phi^n-\phi^{n-1})	 - A \dt \Dh(\phi^{n+1}-\phi^n)\nonumber
                 .
	\end{align}
    \end{small}
The initialization step is as follows:
     \begin{equation}
\phi^1:=\phi^0, \label{scheme-CH_initial}
     \end{equation}
where
     \[
  S'(\phi)=(\frac{1}{\tau}+\frac{1}{N_1})\ln \phi - \rho \ln(1-\rho \phi),\, H'(\phi)=-2\chi\rho\phi,\, \kappa'(\phi)= \frac{2\phi-1}{36\phi^2(1-\phi)^2},
     \]
and $A$ is a positive constant independent on the time step $\dt$ and the spatial mesh step $h$.

Since $\mu$ follows the Laplacian $\Dh$, we omit constants in expressions $S'(\phi)$ and $H'(\phi)$ above.
\begin{remark}
Here adding the extra term $A \dt \Dh(\phi^{n+1}-\phi^n)$ is to guarantee dissipation of the modified discrete energy corresponding to the continuous case in the theoretical level due to the explicitly concave term. In fact, the original discrete energy is numerically non-increasing with time.
In addition, this small term is
$\mathcal{O}(\dt^2)+\mathcal{O}(h^2)$ and there is no adding the extra challenge for the convergent analysis.
\end{remark}
If solutions to the  scheme \eqref{scheme-CH_LOG-1}-\eqref{scheme-CH_initial} exist, it is clear that, for any $n\in\mathbb{N}$ and $n\geq 1$,
     \begin{equation}
\overline{\phi}_{mod}:=|\Omega|^{-1}\ciptwo{\frac{3\phi^{n+1}-\phi^n}{2}}{1} = |\Omega|^{-1}\ciptwo{\frac{3\phi^n-\phi^{n-1}}{2}}{1},\label{mass conservation}
     \end{equation}
     \begin{equation}
\overline{\phi^n}:=|\Omega|^{-1}\ciptwo{\phi^n}{1},\label{mass_define}
     \end{equation}
with $0<\overline{\phi}_{mod} < \nicefrac{1}{\rho}$ and $0<\overline{\phi^n} < \nicefrac{1}{\rho}$.

Thus we obtain
     \[
\ciptwo{\frac{3\phi^{n+1} - \phi^n}{2} - \overline{\phi}_{mod}}{1}=0.
     \]
From the scheme \eqref{scheme-CH_initial}, we have
     \begin{equation}
\overline{\phi}_{0} = \overline{\phi^{1}}.\label{mass conserv-0}
     \end{equation}
Combining \eqref{mass conservation} with \eqref{mass conserv-0}, we get the following mass conservation formula
     \begin{equation}
\overline{\phi}_0 :=  |\Omega|^{-1}\ciptwo{\phi^0}{1} = |\Omega|^{-1}\ciptwo{\phi^1}{1} = \cdots =  |\Omega|^{-1}\ciptwo{\phi^n}{1} = \overline{\phi^n}.\label{mass_conserv-1}
     \end{equation}
 \section{Positivity-preserving property }
 The proof of the following lemma could be found in~\cite{Chen2017A}.
 \begin{lemma}
	\label{CH-positivity-Lem-0}
Suppose that $\phi_1$, $\phi_2 \in \mathcal{C}_{\rm per}$, with $\ciptwo{\phi_1 - \phi_2}{1} = 0$, that is, $\phi_1 - \phi_2\in \mathring{\mathcal{C}}_{\rm per}$, and assume that $\nrm{\phi_1}_\infty < 1$, $\nrm{\phi_2}_\infty \le M$. Then,   we have the following estimate:
     \[
\nrm{\mathcal{L}^{-1} (\phi_1 - \phi_2)}_\infty \le C_1 ,
     \]
where $C_1>0$ depends only upon $M$ and $\Omega$. In particular, $C_1$ is independent of the mesh spacing $h$.
\end{lemma}
Concerning the deGennes coefficient, we find the following lemma.
\begin{lemma}
	\label{CH-positivity-Lem-1}
Assume that  $\phi_1$, $\phi_2 \in (0,1)$, and $\kappa$ is defined by \eqref{CH_energy_K}. Then
     \[
\frac{1}{2} \kappa'(\phi_1) (\phi_2-\phi_1)  \le \kappa(\phi_2) .
     \]
\end{lemma}
\begin{proof}
The proof will be divided into two cases:

Case 1: If $\kappa'(\phi_1) ( \phi_2-\phi_1 ) \le 0$, we see that
\begin{eqnarray}
   \frac12 \kappa'(\phi_1) ( \phi_2-\phi_1 )  \le 0 \le \kappa(\phi_2) ,
	\end{eqnarray}
due to the fact that $\kappa (\phi_2) > 0$, for any $0 < \phi_2 < 1$.
	
Case 2: If $\kappa'(\phi_1) ( \phi_2-\phi_1 ) \ge 0$, we have
\begin{eqnarray}
   \frac12 \kappa'(\phi_1) ( \phi_2-\phi_1 )
   \le \kappa'(\phi_1) ( \phi_2-\phi_1 )  \le \kappa(\phi_2) - \kappa(\phi_1)
   \le \kappa(\phi_2) ,
\end{eqnarray}
in which the second step is based on the convexity of $\kappa(\phi)$ (in terms of $\phi$), and the last step comes from the fact that $\kappa(\phi_1) > 0$.

A combination of these two cases yields the desired result.
\end{proof}

The framework of the following positivity-preserving property of the numerical solution is similar to that in \cite{Chen2017A}.
\begin{theorem}
	\label{CH-positivity}
 Given $\phi^n,\phi^{n-1}\in\mathcal{C}_{\rm per}$, with $0 < \phi^n,\phi^{n-1} < \nicefrac{1}{\rho}$, then $\overline{\phi^n }< \nicefrac{1}{\rho}, \overline{\phi^{n-1}} < \nicefrac{1}{\rho}$, there exists a unique solution $\phi^{n+1}\in\mathcal{C}_{\rm per}$ to the scheme \eqref{scheme-CH_LOG-1}-\eqref{scheme-CH_initial}, with $\overline{\phi^n} = \overline{\phi^{n+1}}$ and  $0 < \phi^{n+1} < \nicefrac{1}{\rho}$.
\end{theorem}
\begin{proof}
First, we define $M=\frac{1}{3}(\overline{\phi^n}+ 2\overline{\phi}_{mod})=\overline{\phi}_0$.
The numerical solution of \eqref{scheme-CH_LOG-1}-\eqref{scheme-mu-0} is a minimizer of the following discrete energy functional:
     \begin{align*}
\mathcal{J}^{n,n-1} (\phi) :=& \frac{1}{12 \dt} \nrm{3\phi-4\phi^n+\phi^{n-1} }_{-1,h}^2 + \frac{1}{\tau}\ciptwo{ \phi }{ \ln \frac{\alpha \phi}{\tau} } + \frac{1}{N_1}\ciptwo{ \phi }{ \ln \frac{\beta \phi}{\tau} }
	\\
&  + \ciptwo{ 1-\rho\phi }{ \ln (1-\rho\phi) } + \ciptwo{ \kappa(\phi) }{a_x((D_x\phi)^2) + a_y((D_y\phi)^2) }
    \\
& - 2\rho\chi\ciptwo{\phi}{2\phi^n-\phi^{n-1}} + \frac{A \dt}{2}\nrm{\nabla_h(\phi-\phi^n)}_2^2,
	 \end{align*}
over the admissible set
	 \[
A_h := \left\{ \phi \in \mathcal{C}_{\rm per} \ \middle| \  0 \le \phi \le \nicefrac{1}{\rho},  \quad  \ciptwo{\phi-M}{1}=0 \right\} \subset \mathbb{R}^{N^2}.
	 \]
It is easy to observe that $\mathcal{J}^{n,n-1}$ is a strictly convex functional over this domain.

To facilitate the analysis below, we transform the minimization problem into an equivalent one
     \begin{eqnarray*}
&&\mathcal{F}^{n,n-1} (\varphi) := \mathcal{J}^{n,n-1} (\varphi + M)
	\\
&=&  \frac{1}{12 \dt} \nrm{3(\varphi+M)-4\phi^n+\phi^{n-1} }_{-1,h}^2
	\\
&& + \frac{1}{\tau}\ciptwo{ \varphi + M }{ \ln \frac{\alpha ( \varphi + M )}{\tau} } + \frac{1}{N_1}\ciptwo{ \varphi + M }{ \ln \frac{\beta ( \varphi + M )}{\tau} }
	\\
&&  + \ciptwo{ 1-\rho( \varphi + M ) }{ \ln (1-\rho( \varphi + M )) }
    \\
&&  + \ciptwo{ \kappa( \varphi + M ) }{a_x((D_x\varphi)^2) + a_y((D_y\varphi)^2) }
	\\
&& - 2\rho\chi\ciptwo{ \varphi + M }{2\phi^n-\phi^{n-1}}
 + \frac{A \dt}{2}\nrm{\nabla_h(\varphi + M-\phi^n)}_2^2,
	 \end{eqnarray*}
defined on the set
	 \[
\mathring{A}_h := \left\{ \varphi \in \mathring{\mathcal{C}}_{\rm per} \ \middle| \  -M \le \varphi \le \nicefrac{1}{\rho}-M  \right\} \subset \mathbb{R}^{N^2}.
	 \]
If $\varphi\in \mathring{A}_h$ minimizes $\mathcal{F}^{n,n-1}$, then $\phi := \varphi + M\in A_h$ minimizes $\mathcal{J}^{n,n-1}$, and \emph{vice versa}. Next, we prove that there exists a minimizer of $\mathcal{F}^{n,n-1}$ over the domain $\mathring{A}_h$. We consider the following closed domain: for $\delta\in (0,\hf)$,
     \[
\mathring{A}_{h,\delta} := \left\{ \varphi \in \mathring{\mathcal{C}}_{\rm per} \ \middle| \  \delta -M \le \varphi \le \nicefrac{1}{\rho}-\delta-M  \right\} \subset \mathbb{R}^{N^2}.
     \]
Since $\mathring{A}_{h,\delta}$ is a bounded, compact, and convex set in the subspace $\mathring{\mathcal{C}}_{\rm per}$, there exists a (not necessarily  unique) minimizer of $\mathcal{F}^{n,n-1}$ over $\mathring{A}_{h, \delta}$. The key point of the positivity analysis is that such a minimizer could not occur on the boundary of $\mathring{A}_{h,\delta}$, if $\delta$ is sufficiently small.

To get a contradiction, suppose that the minimizer of $\mathcal{F}^{n,n-1}$, call it $\varphi^\star$ occurs at a boundary point of $\mathring{A}_{h,\delta}$ and there is at least one grid point $\vec{\alpha}_0 = (i_0,j_0)$ such that $\varphi^\star_{\vec{\alpha}_0}+M = \delta$. Then the grid function $\varphi^\star$ has a global minimum at $\vec{\alpha}_0$. Suppose that $\vec{\alpha}_1 = (i_1,j_1)$ is a grid point at which $\varphi^\star$ achieves its maximum. By the fact that $\overline{\varphi^\star} = 0$, we have $\varphi^\star_{\vec{\alpha}_1}\geq 0$. It is obvious that
     \[
\nicefrac{1}{\rho}-\delta \ge \varphi^\star_{\vec{\alpha}_1}+M \ge M.
	 \]
Since $\mathcal{F}^{n,n-1}$ is smooth over $\mathring{A}_{h,\delta}$, for all $\psi\in \mathring{\mathcal{C}}_{\rm per}$, the directional derivative is
     \begin{align*}
 &d_s \mathcal{F}^{n,n-1}(\varphi^\star+s\psi)|_{s=0} \\
=& \   \ciptwo{\frac{1}{\tau} \ln \frac{\alpha (\varphi^\star+M)}{\tau} + \frac{1}{N_1} \ln \frac{\beta (\varphi^\star+M)}{\tau} - \rho \ln (1- \rho(\varphi^\star+M))}{\psi}
 \\
&+ \ciptwo{\frac{1}{\tau}+\frac{1}{N_1}-\rho}{\psi} +\frac{1}{2 \dt}\ciptwo{ \mathcal{L}^{-1}(3(\varphi^\star+M)-4\phi^n+\phi^{n-1})}{\psi}
   \\
&- \ciptwo{2\rho\chi(2\phi^n-\phi^{n-1})}{\psi}
 - A \dt \ciptwo{\Dh(\varphi^\star +M-\phi^n)}{\psi}
\\
&+ \ciptwo{\kappa'( \varphi^\star + M ) (  a_x((D_x\varphi^\star)^2) + a_y((D_y\varphi^\star)^2)  )}{\psi}
	\\
&+ h^2\sumij \kappa( \varphi^\star_{i,j} + M )\left(D_x \varphi^\star_{i+\hf,j} D_x \psi_{i+\hf,j} + D_x \varphi^\star_{i-\hf,j} D_x \psi_{i-\hf,j}\right)\\
&+ h^2\sumij \kappa( \varphi^\star_{i,j} +M )\left(D_y \varphi^\star_{i,j+\hf} D_y \psi_{i,j+\hf} + D_y \varphi^\star_{i,j-\hf} D_y \psi_{i,j-\hf}\right).
	 \end{align*}
Using the definition of the difference operators $D_x$, $D_y$ and discrete Laplacian operator $\Delta_h$, it is easy to get the following equivalent form
     \begin{align*}
 &d_s \mathcal{F}^{n,n-1}(\varphi^\star+s\psi)|_{s=0}\\
=&  \  \ciptwo{\frac{1}{\tau} \ln \frac{\alpha (\varphi^\star+M)}{\tau} + \frac{1}{N_1} \ln \frac{\beta (\varphi^\star+M)}{\tau} - \rho \ln (1- \rho(\varphi^\star+M))}{\psi}\\
&+ \ciptwo{\frac{1}{\tau}+\frac{1}{N_1}-\rho}{\psi}
+\frac{1}{2 \dt}\ciptwo{ \mathcal{L}^{-1}(3(\varphi^\star+M)-4\phi^n+\phi^{n-1})}{\psi}
   \\
&- \ciptwo{2\rho\chi(2\phi^n-\phi^{n-1})}{\psi}
 - A \dt \ciptwo{\Dh(\varphi^\star + M-\phi^n)}{\psi} \\
&+ \ciptwo{\kappa'( \varphi^\star + M ) (  a_x((D_x\varphi^\star)^2) + a_y((D_y\varphi^\star)^2)  )}{\psi}
	 -  \ciptwo{ \kappa( \varphi^\star + M ) \Delta_h \varphi^\star}{\psi}\\
&+ h\sumij \kappa( \varphi^\star_{i,j} + M )\left(D_x \varphi^\star_{i+\hf,j} \psi_{i+1,j} - D_x \varphi^\star_{i-\hf,j} \psi_{i-1,j}\right)\\
&+ h\sumij \kappa( \varphi^\star_{i,j} + M )\left(D_y \varphi^\star_{i,j+\hf} \psi_{i,j+1} - D_y \varphi^\star_{i,j-\hf} \psi_{i,j-1}\right).
	 \end{align*}
This time, let us pick the direction $\psi \in \mathring{\mathcal{C}}_{\rm per}$, such that
	 \[
\psi_{i,j} = \delta_{i,i_0}\delta_{j,j_0} - \delta_{i,i_1}\delta_{j,j_1} .
	 \]
Then the derivative may be expressed as
     \begin{align}
&\frac{1}{h^2} d_s \mathcal{F}^{n,n-1}(\varphi^\star+s\psi)|_{s=0}
    \nonumber
    \\
 =& \left(\frac{1}{\tau} \ln \frac{\alpha (\varphi^\star_{\vec{\alpha}_0}+M)}{\tau} + \frac{1}{N_1} \ln \frac{\beta (\varphi^\star_{\vec{\alpha}_0}+M)}{\tau}
- \rho \ln (1- \rho(\varphi^\star_{\vec{\alpha}_0}+M))\right)
	\nonumber
    \\
 -& \left(\frac{1}{\tau} \ln \frac{\alpha (\varphi^\star_{\vec{\alpha}_1}+M)}{\tau} + \frac{1}{N_1} \ln \frac{\beta (\varphi^\star_{\vec{\alpha}_1}+M)}{\tau} - \rho \ln (1- \rho(\varphi^\star_{\vec{\alpha}_1}+M))\right)
	\nonumber
    \\
-& 2\rho\chi \left( 2(\phi^n_{\vec{\alpha}_0} - \phi^n_{\vec{\alpha}_1})- (\phi^{n-1}_{\vec{\alpha}_0} - \phi^{n-1}_{\vec{\alpha}_1})\right)
- A \dt (\Dh\varphi^\star_{\vec{\alpha}_0} - \Dh\varphi^\star_{\vec{\alpha}_1}) + A \dt (\Dh\phi^n_{\vec{\alpha}_0} - \Dh\phi^n_{\vec{\alpha}_1})
	\nonumber
    \\
-& \left(\kappa( \varphi^\star_{\vec{\alpha}_0} + M ) \Delta_h \varphi^\star_{\vec{\alpha}_0} -  \kappa( \varphi^\star_{\vec{\alpha}_1} + M ) \Delta_h \varphi^\star_{\vec{\alpha}_1} \right)
	\nonumber
	\\
+&\frac{1}{2 \dt} \bigg(\mathcal{L}^{-1}(3(\varphi^\star+M)-4\phi^n+\phi^{n-1})_{\vec{\alpha}_0} -\mathcal{L}^{-1}(3(\varphi^\star+M)-4\phi^n+\phi^{n-1})_{\vec{\alpha}_1} \bigg)
    \nonumber
	\\
+&\kappa'( \varphi^\star_{\vec{\alpha}_0} + M )
( a_x((D_x\varphi^\star_{\vec{\alpha}_0})^2) + a_y((D_y\varphi^\star_{\vec{\alpha}_0})^2) )
 - \kappa'( \varphi^\star_{\vec{\alpha}_1} + M )
 ( a_x((D_x\varphi^\star_{\vec{\alpha}_1})^2) + a_y((D_y\varphi^\star_{\vec{\alpha}_1})^2) )
	\nonumber
    \\
+& \frac{1}{h} \left(\kappa( \varphi^\star_{i_0-1,j_0} + M ) D_x \varphi^\star_{i_0-\hf,j_0}
  -  \kappa( \varphi^\star_{i_0+1,j_0} + M ) D_x \varphi^\star_{i_0+\hf,j_0}\right)
	\nonumber
    \\
+& \frac{1}{h} \left(\kappa( \varphi^\star_{i_0,j_0-1} + M ) D_y \varphi^\star_{i_0,j_0-\hf}
  -  \kappa( \varphi^\star_{i_0,j_0+1} + M ) D_y \varphi^\star_{i_0,j_0+\hf}\right)
	\nonumber
    \\
-& \frac{1}{h} \left(\kappa( \varphi^\star_{i_1-1,j_1} + M ) D_x \varphi^\star_{i_1-\hf,j_1}
  -  \kappa( \varphi^\star_{i_1+1,j_1} + M ) D_x \varphi^\star_{i_1+\hf,j_1}\right)
	\nonumber
    \\
-& \frac{1}{h} \left(\kappa( \varphi^\star_{i_1,j_1-1} +M ) D_y \varphi^\star_{i_1,j_1-\hf}
  -  \kappa( \varphi^\star_{i_1,j_1+1} + M ) D_y \varphi^\star_{i_1,j_1+\hf}\right).
	\label{CH_LOG-positive-4}
	\end{align}
For simplicity, now let us write $\phi^\star := \varphi^\star +M$. Since  $\phi^\star_{\vec{\alpha}_0} =  \delta$ and $\phi^\star_{\vec{\alpha}_1} \ge M$, we have
     \begin{equation}
  \frac{1}{\tau} \ln \frac{\alpha \phi^\star_{\vec{\alpha}_0}}{\tau} + \frac{1}{N_1} \ln \frac{\beta \phi^\star_{\vec{\alpha}_0}}{\tau}
- \rho \ln (1- \rho \phi^\star_{\vec{\alpha}_0}) =   \frac{1}{\tau} \ln \frac{\alpha \delta}{\tau} + \frac{1}{N_1} \ln \frac{\beta \delta}{\tau}
- \rho \ln (1- \rho \delta),
	\label{CH_LOG-positive-5}
     \end{equation}
     \begin{equation}
    \frac{1}{\tau} \ln \frac{\alpha \phi^\star_{\vec{\alpha}_1}}{\tau} + \frac{1}{N_1} \ln \frac{\beta \phi^\star_{\vec{\alpha}_1}}{\tau}
- \rho \ln (1- \rho \phi^\star_{\vec{\alpha}_1}) \ge   \frac{1}{\tau} \ln \frac{\alpha M}{\tau} + \frac{1}{N_1} \ln \frac{\beta M}{\tau}
- \rho \ln (1- \rho M).
     \end{equation}
Since $\phi^\star$ takes a  minimum at the grid point $\vec{\alpha}_0$, with $\phi^\star_{\vec{\alpha}_0} =  \delta \le \phi^\star_{i,j}$, for any $(i,j)$, and a maximum at the grid point $\vec{\alpha}_1$, with $\phi^\star_{\vec{\alpha}_1} \ge \phi^\star_{i,j}$, for any $(i,j)$,
     \begin{equation}
\Delta_h \phi^\star_{\vec{\alpha}_0} \ge 0 ,  \quad \Delta_h \phi^\star_{\vec{\alpha}_1} \le 0 .
	\label{CH_LOG-positive-6}
	\end{equation}
For the numerical solution $\phi^n$ at the previous time step, the priori assumption $0< \phi^n, \phi^{n-1} < \nicefrac{1}{\rho}$ indicates that
     \begin{equation}
\label{CH_LOG-positive-8-0}
-\frac{1}{\rho} < \phi^n_{\vec{\alpha}_0} - \phi^n_{\vec{\alpha}_1} < \frac{1}{\rho}
,\quad -\frac{1}{\rho} < \phi^{n-1}_{\vec{\alpha}_0} - \phi^{n-1}_{\vec{\alpha}_1} < \frac{1}{\rho}
	 \end{equation}
     \begin{equation}
\label{CH_LOG-positive-8-1}
-\frac{8}{\rho h^2} \le \Dh\phi^n_{\vec{\alpha}_0} - \Dh\phi^n_{\vec{\alpha}_1} \le \frac{8}{\rho h^2}.
     \end{equation}
For the seventh term appearing in (\ref{CH_LOG-positive-4}), we apply Lemma~\ref{CH-positivity-Lem-0} and obtain
     \begin{small}
	 \begin{equation}
-  C_1 \dt^{-1} \le\frac{1}{2 \dt} \left(\mathcal{L}^{-1}(3\phi^\star-4\phi^n+\phi^{n-1})_{\vec{\alpha}_0} -
 \mathcal{L}^{-1}(3\phi^\star-4\phi^n+\phi^{n-1})_{\vec{\alpha}_1} \right)  \le  C_1 \dt^{-1}.
	\label{CH_LOG-positive-9}
	 \end{equation}
     \end{small}
For the eighth, tenth and eleventh terms appearing in (\ref{CH_LOG-positive-4}) are non-positive.\\
For the ninth and the last two terms appearing in (\ref{CH_LOG-positive-4}), we apply Lemma~\ref{CH-positivity-Lem-1} and know that they are non-positive together.
Consequently,  a substitution of (\ref{CH_LOG-positive-5})-(\ref{CH_LOG-positive-9}) into (\ref{CH_LOG-positive-4}) yields the following bound on the directional derivative:
     \begin{eqnarray*}
&&\frac{1}{h^2} d_s \mathcal{F}^{n,n-1}(\varphi^\star+s\psi)|_{s=0}\\
&\le& \left(\frac{1}{\tau} \ln \frac{\alpha \delta}{\tau} + \frac{1}{N_1} \ln \frac{\beta \delta}{\tau}
- \rho \ln (1- \rho \delta)\right)+ 6\chi + \frac{8 A}{\rho} \frac{\dt}{h^2} +  C_1 \dt^{-1}
    \\
&&-\left(\frac{1}{\tau} \ln \frac{\alpha M}{\tau} + \frac{1}{N_1} \ln \frac{\beta M}{\tau} - \rho \ln (1- \rho M)\right)
    \\
&=& \left((\frac{1}{\tau} + \frac{1}{N_1}) \ln \delta - \rho \ln (1- \rho \delta)\right)+ 6\chi  + \frac{8 A}{\rho} \frac{\dt}{h^2} +  C_1 \dt^{-1}
    \\
&&- \left((\frac{1}{\tau} + \frac{1}{N_1}) \ln M - \rho \ln (1- \rho M)\right).
     \end{eqnarray*}
We denote $C_2 = 6\chi  + \frac{8 A}{\rho} \frac{\dt}{h^2} +  C_1 \dt^{-1}$. Note that $C_2$ is a constant for the fixed $\dt$ and $h$, though it becomes singular as $\dt \to 0$ or $h \to 0$. However, for any fixed $\dt$ and $h$, we may choose $\delta\in(0,\hf)$ sufficiently small so that
	 \begin{equation}
	 \small
\left((\frac{1}{\tau} + \frac{1}{N_1}) \ln \delta - \rho \ln (1- \rho \delta)\right)
- \left((\frac{1}{\tau} + \frac{1}{N_1}) \ln M - \rho \ln (1- \rho M)\right) + C_2 < 0 .  \label{CH_LOG-positive-11}
	 \end{equation}
This in turn shows that, provided $\delta$ satisfies (\ref{CH_LOG-positive-11}),
     \[
\frac{1}{h^2} d_s \mathcal{F}^{n,n-1}(\varphi^\star+s\psi)|_{s=0} < 0 .
     \]
As before, this contradicts the assumption that $\mathcal{F}^{n,n-1}$ has a minimum at $\varphi^\star$, since the directional derivative is negative in a direction pointing into the interior of $\mathring{A}_{h,\delta}$.

Using very similar arguments, we can also prove that the global minimum of $\mathcal{F}^{n,n-1}$ over $\mathring{A}_{h,\delta}$ could not occur at a boundary point $\varphi^\star$ such that  $\varphi^\star_{\vec{\alpha}_0}+M  = \nicefrac{1}{\rho}-\delta$, for some $\vec{\alpha}_0$, so that the grid function $\varphi^\star$ has a global maximum at $\vec{\alpha}_0$. The details are left to interested readers.

A combination of these two facts shows that, the global minimum of $\mathcal{F}^{n,n-1}$ over $\mathring{A}_{h,\delta}$ could only possibly occur at interior point $\varphi\in (\mathring{A}_{h,\delta})^{\rm o}\subset (\mathring{A}_h)^{\rm o}$. We conclude that there must be a solution $\phi = \varphi+M\in A_h$ that minimizes $\mathcal{J}^{n,n-1}$ over $A_h$, which is equivalent to the numerical solution of (\ref{scheme-CH_LOG-1})-(\ref{scheme-mu-0}). The existence of the numerical solution is established.

In addition, since $\mathcal{J}^{n,n-1}$ is a strictly convex function over $A_h$, the uniqueness analysis for this numerical solution is straightforward. The proof of Theorem~\ref{CH-positivity} is complete.
\end{proof}
\section{Unconditional energy stability} \label{sec:energy stability}

\begin{theorem}
    For $n\geq 1$, we define the modified discrete energy as
     \[
E_h (\phi^{n+1},\phi^n) := F(\phi^{n+1}) + \frac{1}{4\dt} \nrm{\phi^{n+1}-\phi^n}_{-1,h}^2 + \chi \rho \nrm{\phi^{n+1}-\phi^n}_2^2,
     \]
and suppose $A\geq \chi^2 \rho^2$. Then the numerical scheme \eqref{scheme-CH_LOG-1}-\eqref{scheme-mu-0} has the energy-decay property
     \begin{equation}
E_h (\phi^{n+1},\phi^n) + \dt \left(1 - \frac{\chi^2 \rho^2}{A}\right) \nrm{ \frac{\phi^{n+1}-\phi^n}{\dt} }_{-1,h}^2 \leq E_h (\phi^n,\phi^{n-1}).	\label{CH-discrete energy}
     \end{equation}
\end{theorem}

\begin{proof}
Due to the mass conservation,  $ \mathcal{L}^{-1} (\phi^{n+1}-\phi^n)$ is well-defined. Taking a discrete inner product with (\ref{scheme-CH_LOG-1}) by $ \mathcal{L}^{-1} (\phi^{n+1}-\phi^n)$ , with (\ref{scheme-mu-0}) by $\phi^{n+1}-\phi^n$ yields
     \begin{align*}
0 =& \frac{1}{2\dt} \ciptwo{ 3\phi^{n+1} - 4\phi^n + \phi^{n-1} }{ \mathcal{L}^{-1} (\phi^{n+1}-\phi^n) }\\
+& \ciptwo{ \delta_\phi F_c(\phi^{n+1}) - \delta_\phi F_e(2\phi^n-\phi^{n-1}) }{ \phi^{n+1}-\phi^n } \\
+&  A \dt \nrm{\nabla_h(\phi^{n+1}-\phi^n)}_2^2 .
	 \end{align*}
The equivalent form is the following identity
     \begin{eqnarray}
0 &=& \frac{1}{2\dt} \ciptwo{ 3\phi^{n+1} - 4\phi^n + \phi^{n-1} }{ \mathcal{L}^{-1} (\phi^{n+1}-\phi^n) } \nonumber\\
&+& \ciptwo{ \delta_\phi F_c(\phi^{n+1}) - \delta_\phi F_e(\phi^n) }{ \phi^{n+1}-\phi^n }\nonumber
 \\
&-& \ciptwo{ \delta_\phi F_e(\phi^n) - \delta_\phi F_e(\phi^{n-1}) }{ \phi^{n+1}-\phi^n } \nonumber\\
&+& A \dt \nrm{\nabla_h(\phi^{n+1}-\phi^n)}_2^2
\label{CH_BDF} .
	 \end{eqnarray}
For the first term of the right hand side of \eqref{CH_BDF}, we have
     \begin{align}
&\frac{1}{2\dt} \ciptwo{ 3\phi^{n+1} - 4\phi^n + \phi^{n-1} }{ \mathcal{L}^{-1} (\phi^{n+1}-\phi^n) } \nonumber\\
=& \dt \left(\frac{5}{4}\nrm{\frac{\phi^{n+1}-\phi^n}{\dt}}_{-1,h}^2 - \frac{1}{4}\nrm{\frac{\phi^n-\phi^{n-1}}{\dt}}_{-1,h}^2\right)
\nonumber\\
&+ \frac{\dt^3}{4}\nrm{\frac{\phi^{n+1}-2\phi^n+\phi^{n-1}}{\dt^2}}_{-1,h}^2.
     \end{align}
For the second term of the right hand side of \eqref{CH_BDF}, we have
     \begin{equation}
\ciptwo{ \delta_\phi F_c(\phi^{n+1}) - \delta_\phi F_e(\phi^n) }{ \phi^{n+1}-\phi^n } \geq F(\phi^{n+1})-F(\phi^n).
     \end{equation}
For the third term of the right hand side of \eqref{CH_BDF}, we have
     \begin{eqnarray}
&&-\ciptwo{ \delta_\phi F_e(\phi^n) - \delta_\phi F_e(\phi^{n-1}) }{ \phi^{n+1}-\phi^n } \nonumber\\
&=&
-2 \chi \rho \ciptwo{\phi^n-\phi^{n-1}}{\phi^{n+1}-\phi^n}\nonumber\\
&=& - \chi \rho \left(\nrm{ \phi^n-\phi^{n-1} }_2^2 - \nrm{ \phi^{n+1}-2\phi^n+\phi^{n-1} }_2^2 + \nrm{ \phi^{n+1}-\phi^n }_2^2\right) .
     \end{eqnarray}
Going back \eqref{CH_BDF} and by simple calculation,  we arrive at
     \begin{eqnarray}
&&F(\phi^{n+1})-F(\phi^n) + \dt \left(\frac{5}{4}\nrm{\frac{\phi^{n+1}-\phi^n}{\dt}}_{-1,h}^2 - \frac{1}{4}\nrm{\frac{\phi^n-\phi^{n-1}}{\dt}}_{-1,h}^2\right)  \nonumber\\
&&- \chi \rho \left(\nrm{ \phi^n-\phi^{n-1} }_2^2  - \nrm{ \phi^{n+1}-\phi^n }_2^2\right) + A \dt \nrm{\nabla_h(\phi^{n+1}-\phi^n)}_2^2  \nonumber\\
&&\leq 2 \chi \rho  \nrm{ \phi^{n+1}-\phi^n }_2^2
\label{CH_BDF2}.
     \end{eqnarray}
For the right side of \eqref{CH_BDF2}, we have
     \begin{eqnarray}
\nrm{\phi^{n+1}-\phi^n}_2^2 &=& \nrm{\nabla_h(\phi^{n+1}-\phi^n)}_2 \cdot \nrm{\phi^{n+1}-\phi^n}_{-1,h}
   \nonumber
   \\
   &\leq& \frac{\dt}{2\alpha}\nrm{\nabla_h(\phi^{n+1}-\phi^n)}_2^2 + \frac{\alpha\dt}{2}\nrm{\frac{\phi^{n+1}-\phi^n}{\dt}}_{-1,h}^2.
     \end{eqnarray}
At last, we get
     \begin{eqnarray*}
&&F(\phi^{n+1}) + \chi \rho \nrm{ \phi^{n+1}-\phi^n }_2^2 + \frac{\dt}{4}\nrm{\frac{\phi^{n+1}-\phi^n}{\dt}}_{-1,h}^2
   \\
&&+ \dt (A-\frac{\chi \rho}{\alpha}) \nrm{\nabla_h(\phi^{n+1}-\phi^n)}_2^2 + \dt(1-\alpha \chi\rho)\nrm{\frac{\phi^{n+1}-\phi^n}{\dt}}_{-1,h}^2
   \\
&&\leq
F(\phi^n) + \chi \rho \nrm{ \phi^n-\phi^{n-1} }_2^2 + \frac{\dt}{4}\nrm{\frac{\phi^n-\phi^{n-1}}{\dt}}_{-1,h}^2.
     \end{eqnarray*}
Let $\alpha = \frac{\chi \rho}{A}$, when $A \geq \chi^2 \rho^2$, we have
     \begin{equation*}
    E_h (\phi^{n+1},\phi^n) + \dt \left(1 - \frac{\chi^2 \rho^2}{A}\right) \nrm{ \frac{\phi^{n+1}-\phi^n}{\dt} }_{-1,h}^2 \leq E_h (\phi^n,\phi^{n-1}).
    \end{equation*}

    \end{proof}
	\section{Optimal rate convergence analysis in $\ell^\infty (0,T; H_h^{-1}) \cap \ell^2 (0,T; H_h^1)$}
	\label{sec:convergence}
Let $\Phi$ be the exact solution for the $H^{-1}$ flow \eqref{CH equation-0}. With the initial data with sufficient regularity, we could assume that the exact solution has regularity of class $\mathcal{R}$:
     \begin{equation}
\Phi \in \mathcal{R} := H^2 \left(0,T; C_{\rm per}(\Omega)\right) \cap L^\infty \left(0,T; C^6_{\rm per}(\Omega)\right).
	\label{assumption:regularity.1}
     \end{equation}
Define $\Phi_N (\, \cdot \, ,t) := {\cal P}_N \Phi (\, \cdot \, ,t)$, the (spatial) Fourier projection of the exact solution into ${\cal B}^m$, the space of trigonometric polynomials of degree to and including  $K$.  The following projection approximation is standard:\\
if $\Phi\in L^\infty(0,T;H^\ell_{\rm per}(\Omega))$, for some $\ell\in\mathbb{N}$,
     \[
\nrm{\Phi_N - \Phi}_{L^\infty(0,T;H^k)}
   \le C h^{\ell-k} \nrm{\Phi }_{L^\infty(0,T;H^\ell)},  \quad  \ 0 \le k \le \ell .
	\label{projection-est-0}
     \]
By $\Phi_N^m$, $\Phi^m$ we denote $\Phi_N(\, \cdot \, , t_m)$ and $\Phi(\, \cdot \, , t_m)$, respectively, with $t_m = m\cdot \dt$. Since $\Phi_N \in {\cal B}^m$, the mass conservative property is available at the discrete level:
     \[
\overline{\Phi_N^m} = \frac{1}{|\Omega|}\int_\Omega \, \Phi_N ( \cdot, t_m) \, d {\bf x} = \frac{1}{|\Omega|}\int_\Omega \, \Phi_N ( \cdot, t_{m+1}) \, d {\bf x} = \overline{\Phi_N^{m+1}} ,  \quad \ m \in\mathbb{N}.
	\label{mass conserv-1}
     \]
On the other hand, the solution of (\ref{scheme-CH_LOG-1})-(\ref{scheme-mu-0}) is also mass conservative at the discrete level:
     \begin{equation}
\overline{\phi^m} = \overline{\phi^{m+1}} ,  \quad  \ m \in \mathbb{N} .
	\label{mass conserv-2}
     \end{equation}
As indicated before, we use the mass conservative projection for the initial data:  $\phi^0 = {\mathcal P}_h \Phi_N (\, \cdot \, , t=0)$, that is
     \[
\phi^0_{i,j} := \Phi_N (p_i, p_j, t=0).
	\label{initial data-0}
     \]
The error grid function is defined as
	\begin{equation}
\tilde{\phi}^m := \mathcal{P}_h \Phi_N^m - \phi^m ,  \quad  \ m \in \left\{ 0 ,1 , 2, 3, \cdots \right\} .
	\label{CH_LOG-error function-1}
	\end{equation}
Therefore, it follows that  $\overline{\tilde{\phi}^m} =0$, for any $m \in \left\{ 0 ,1 , 2, 3, \cdots \right\}$,  so that the discrete norm $\nrm{ \, \cdot \, }_{-1,h}$ is well defined for the error grid function.
Before proceeding into the convergence analysis, we introduce a new norm. Let $\Omega$ be an arbitrary bounded domain and $\mathbf{p} = [u,v]^T \in [L^2(\Omega)]^2$. We define $\nrm{\, \cdot \, }_{-1,G}$ to be a weighted inner product
     \begin{equation}
\nrm{ \mathbf{p} }_{-1,G}^2 = (\mathbf{p},G(-\Delta_h)^{-1}\mathbf{p}),\quad G = \left(
                                                                  \begin{array}{cc}
                                                                    \frac{1}{2} & -1 \\
                                                                    -1 & \frac{5}{2} \\
                                                                  \end{array}
                                                                \right).	
                                                                \nonumber
     \end{equation}
Since G is symmetric positive definite, the norm is well-defined. Moreover,
     \[
G = \left(
      \begin{array}{cc}
        \frac{1}{2} & -1 \\
        -1 & \frac{5}{2} \\
      \end{array}
    \right)
  =  \left(
      \begin{array}{cc}
        \frac{1}{2} & -1 \\
        -1 & 2 \\
      \end{array}
    \right)
  + \left(
      \begin{array}{cc}
        0 & 0 \\
        0 & \frac{1}{2} \\
      \end{array}
    \right)
  =:G_1 + G_2.
     \]
By the positive semi-definiteness of $G_1$, we immediately have
     \[
\nrm{ \mathbf{p} }_{-1,G}^2 = (\mathbf{p},(G_1+G_2)(-\Delta_h)^{-1}\mathbf{p}) \geq (\mathbf{p},G_2(-\Delta_h)^{-1}\mathbf{p}) = \frac{1}{2} \nrm{v}_{-1,h}^2.
     \]
In addition, for any $v_i \in L^2(\Omega),i = 0,1,2$, the following equality is valid:
     \begin{equation}
\left(\frac{3}{2}v_2-2v_1+\frac{1}{2}v_0, (-\Delta_h)^{-1}v_2\right) = \frac{1}{2}(\nrm{ \mathbf{p_2} }_{-1,G}^2 - \nrm{ \mathbf{p_1} }_{-1,G}^2) + \frac{\nrm{v_2-2v_1+v_0 }_{-1,h}^2}{4},
     \end{equation}
with $\mathbf{p_1} = [v_0,v_1]^T, \mathbf{p_2} = [v_1,v_2]^T$.
\begin{theorem}
	\label{thm:convergence}
Given initial data $\Phi(\, \cdot \, ,t=0) \in C^6_{\rm per}(\Omega)$, suppose the exact solution for the equation \eqref{CH equation-0} is of regularity class $\mathcal{R}$ in \eqref{assumption:regularity.1}. Then, provided that $\dt$ and $h$ are sufficiently small, for all positive integers $n$, such that $t_n=n \dt \le T$, we have
     \[
\nrm{ \tilde{\phi}^n }_{-1,h} +\left(\frac{\dt}{18}\sum_{k=1}^n\nrm{ \nabla_h \tilde{\phi}^k }_2^2\right)^{\nicefrac{1}{2}}\le C ( \dt^2 + h^2 ),
	\label{CH_LOG-convergence-0}
     \]
where $C>0$ is independent of $n$, $\dt$, and $h$.
	\end{theorem}
\begin{proof}
A careful consistent analysis indicates the following truncation error estimate:
\begin{eqnarray}
   \frac{3\Phi_N^{n+1} - 4\Phi_N^n + \Phi_N^{n-1}}{2 \dt}  &=& \Delta_h
 \Bigl( \delta_\phi F_S(\Phi_N^{n+1}) + \delta_\phi F_{K_1}(\Phi_N^{n+1})+ \delta_\phi F_{K_2}(\Phi_N^{n+1})
    \nonumber
\\
  &&  \quad
  - \delta_\phi F_H(2\Phi_N^n-\Phi_N^{n-1}) - A \dt \Dh(\Phi_N^{n+1}-\Phi_N^n) \Bigr)
      \nonumber
\\
  &&  \quad
  + \tau^n ,
	\label{CH_LOG-consistency-1}
	\end{eqnarray}
with $\| \tau^n \|_{-1,h} \le C (\dt^2 + h^2)$. Observe that in~\eqref{CH_LOG-consistency-1}, and from this point forward, we drop the operator $\mathcal{P}_h$, which should appear in front of $\Phi_N$, for simplicity.

Subtracting the numerical schemes \eqref{scheme-CH_LOG-1}-\eqref{scheme-mu-0} from (\ref{CH_LOG-consistency-1}) gives
      \begin{align}
   \frac{3\tilde{\phi}^{n+1} - 4\tilde{\phi}^n +\tilde{\phi}^{n-1}}{2\dt}
=& \Delta_h
 \Bigl( (\delta_\phi F_S(\Phi_N^{n+1}) - \delta_\phi F_S(\phi^{n+1})) + (\delta_\phi F_{K_1}(\Phi_N^{n+1})- \delta_\phi F_{K_1}(\phi^{n+1}))
  \nonumber
\\
  &  \quad \quad
  + (\delta_\phi F_{K_2}(\Phi_N^{n+1})- \delta_\phi F_{K_2}(\phi^{n+1}))  - A \dt \Dh(\tilde{\phi}^{n+1}-\tilde{\phi}^n)
  \nonumber
\\
  & \quad \quad
-  (\delta_\phi F_H(2\Phi_N^n-\Phi_N^{n-1})- \delta_\phi F_H(2\phi^n-\phi^{n-1}))  \Bigr)
    + \tau^n .
	\label{CH_LOG-consistency-2}
	  \end{align}
Since the numerical error function has zero-mean, we see that $\mathcal{L}^{-1} \tilde{\phi}^n$ is well-defined, for any $n \ge 0$. Taking a discrete inner product with \eqref{CH_LOG-consistency-2} by $ \mathcal{L}^{-1} \tilde{\phi}^{n+1}$ yields
      \begin{eqnarray}
&&\nrm{ \mathbf{p}^{n+1} }_{-1,G}^2 - \nrm{ \mathbf{p}^n }_{-1,G}^2 + \frac{1}{2}\nrm{\tilde{\phi}^{n+1} - 2\tilde{\phi}^n+ \tilde{\phi}^{n-1}}_{-1,h}^2
	\nonumber
	\\
&& +2 \dt \ciptwo{\delta_\phi F_S(\Phi_N^{n+1}) - \delta_\phi F_S(\phi^{n+1}) }{ \tilde{\phi}^{n+1} }
	\nonumber
	\\
&&+ 2 \dt \ciptwo{\delta_\phi F_{K_1}(\Phi_N^{n+1})- \delta_\phi F_{K_1}(\phi^{n+1})  }{ \tilde{\phi}^{n+1} }
	\nonumber
	\\
&&+ 2 \dt \ciptwo{ \delta_\phi F_{K_2}(\Phi_N^{n+1})- \delta_\phi F_{K_2}(\phi^{n+1}) }{ \tilde{\phi}^{n+1} }
	\nonumber
	\\
&&- 2 A \dt^2\ciptwo{\Dh(\tilde{\phi}^{n+1}-\tilde{\phi}^n)}{\tilde{\phi}^{n+1}}
	\nonumber
	\\
&=&  4 \chi \rho \dt \ciptwo{ 2\tilde{\phi}^n - \tilde{\phi}^{n-1} }{ \tilde{\phi}^{n+1} }  + 2 \dt \ciptwo{ \tau^n }{ \mathcal{L}^{-1}\tilde{\phi}^{n+1} },
	\label{CH_LOG-convergence-1}
	  \end{eqnarray}
where $\mathbf{p}^{n+1} = ( \tilde{\phi}^n, \tilde{\phi}^{n+1})$.\\
The estimate for the term associated with $F_{K_2}$ is straightforward:
\begin{eqnarray}
&2 \dt \ciptwo{ \delta_\phi F_{K_2}(\Phi_N^{n+1})- \delta_\phi F_{K_2}(\phi^{n+1}) }{ \tilde{\phi}^{n+1} } \\ \nonumber
&= 2 \dt \ciptwo{ - \frac{1}{18}\Delta_h \tilde{\phi}^{n+1}} { \tilde{\phi}^{n+1}}  = \frac{1}{9} \dt \| \nabla_h \tilde{\phi}^{n+1} \|_2^2 .  \label{CH_LOG-convergence-2}
\end{eqnarray}
For the $F_S$ and $F_{K_1}$ terms, the fact that $F_S$ and $F_{K_1}$ are both convex yields the following result:
\begin{eqnarray}
 \ciptwo{\delta_\phi F_S(\Phi_N^{n+1}) - \delta_\phi F_S(\phi^{n+1}) }{ \tilde{\phi}^{n+1} }&\ge& 0 ,\\
	\label{CH_LOG-convergence-3-1}
\ciptwo{\delta_\phi F_{K_1}(\Phi_N^{n+1})- \delta_\phi F_{K_1}(\phi^{n+1})  }{ \tilde{\phi}^{n+1} }  &\ge& 0.
	\label{CH_LOG-convergence-3-2}
	\end{eqnarray}
For the artificial term, we have
\begin{equation}
2\ciptwo{\nabla_h(\tilde{\phi}^{n+1}-\tilde{\phi}^n)}{\nabla_h\tilde{\phi}^{n+1}} = \nrm{\nabla_h\tilde{\phi}^{n+1}}_2^2 -\nrm{\nabla_h\tilde{\phi}^n}_2^2 + \nrm{\nabla_h(\tilde{\phi}^{n+1}-\tilde{\phi}^n)}_2^2.
	\label{CH_LOG-convergence-4}
\end{equation}
For the inner product associated with the concave part, the following estimate is derived:
\begin{small}
	\begin{eqnarray}
 &&4 \chi \rho \ciptwo{ 2\tilde{\phi}^n - \tilde{\phi}^{n-1} }{ \tilde{\phi}^{n+1} }
 	\nonumber
	\\
  &=& 8\chi \rho \ciptwo{ \tilde{\phi}^n }{ \tilde{\phi}^{n+1} }  -  4 \chi \rho \ciptwo{\tilde{\phi}^{n-1} }{ \tilde{\phi}^{n+1} }
	\nonumber
	\\
 &\le& 8 \chi \rho \nrm{ \tilde{\phi}^n }_{-1,h}  \nrm{ \nabla_h \tilde{\phi}^{n+1} }_2
 + 4 \chi \rho \nrm{ \tilde{\phi}^{n-1} }_{-1,h}  \nrm{ \nabla_h \tilde{\phi}^{n+1} }_2
	\\
& \le& 32 \chi^2 \rho^2 \varepsilon_1^{-2} \nrm{ \tilde{\phi}^n }_{-1,h}^2  + \frac{\varepsilon_1^2}{2} \nrm{ \nabla_h \tilde{\phi}^{n+1} }_2^2 + 8 \chi^2 \rho^2 \varepsilon_2^{-2} \nrm{ \tilde{\phi}^{n-1} }_{-1,h}^2  + \frac{\varepsilon_2^2}{2} \nrm{ \nabla_h \tilde{\phi}^{n+1} }_2^2	\nonumber.
	\label{CH_LOG-convergence-5}
	\end{eqnarray}
\end{small}
The term associated with the local truncation error can be controlled in a standard way:
\begin{small}
	\begin{equation}
2 \ciptwo{\tau^n }{ \mathcal{L}^{-1}\tilde{\phi}^{n+1} }  \le  2 \nrm{ \tau^n }_{-1,h}  \nrm{ \tilde{\phi}^{n+1} }_{-1,h}  \le  2 \varepsilon^{-2} \nrm{ \tau^n }_{-1,h}^2 + \frac{\varepsilon^2}{2}  \nrm{ \tilde{\phi}^{n+1} }_{-1,h}^2 .
	\label{CH_LOG-convergence-6}
	\end{equation}
\end{small}
Going back to \eqref{CH_LOG-convergence-1}, when $n \geq 1$, we arrive

	\begin{align}
&\nrm{ \mathbf{p}^{n+1} }_{-1,G}^2 - \nrm{ \mathbf{p}^n }_{-1,G}^2 	
 + A \dt^2\nrm{\nabla_h\tilde{\phi}^{n+1}}_2^2 - A \dt^2\nrm{\nabla_h\tilde{\phi}^n}_2^2
+ \frac{\dt}{9}\nrm{ \nabla_h \tilde{\phi}^{n+1} }_2^2
	\nonumber
    \\
\leq &\frac{32\chi^2 \rho^2}{\varepsilon_1^2}\dt\nrm{ \tilde{\phi}^n }_{-1,h}^2 + \frac{8\chi^2 \rho^2}{\varepsilon_2^2}\dt\nrm{ \tilde{\phi}^{n-1} }_{-1,h}^2
+\frac{\varepsilon^2}{2}\dt \nrm{ \tilde{\phi}^{n+1} }_{-1,h}^2
\nonumber
\\
+&(\frac{\varepsilon_1^2}{2}+\frac{\varepsilon_2^2}{2}) \dt \nrm{ \nabla_h \tilde{\phi}^{n+1} }_2^2
 + \frac{2 \dt}{\varepsilon^2}\nrm{ \tau^n }_{-1,h}^2.
\label{CH_LOG-convergence-7}
	\end{align}
Now we observe that $\nrm{ \mathbf{p}^1 }_{-1,G}^2 = \frac{5}{2 }\nrm{ \tilde{\phi}^1 }_{-1,h}^2=\frac{5}{2 }\nrm{ \tilde{\phi}^0 }_{-1,h}^2=0.$\\
Summing both sides of \eqref{CH_LOG-convergence-7} with respect to $n$ gives
     \begin{eqnarray}
&&\nrm{\mathbf{p}^{n+1}}_{-1,G}^2 +  A \dt^2\nrm{\nabla_h\tilde{\phi}^{n+1}}_2^2
+ \frac{\dt}{9}\sum_{k=1}^n\nrm{ \nabla_h \tilde{\phi}^{k+1} }_2^2
\\
\nonumber
&\leq& \frac{\varepsilon^2}{2}\dt \nrm{ \tilde{\phi}^{n+1} }_{-1,h}^2
+\left(\frac{32\chi^2 \rho^2}{\varepsilon_1^2} + \frac{8\chi^2 \rho^2}{\varepsilon_2^2} + \frac{\varepsilon^2}{2}\right) \dt \sum_{k=1}^n \nrm{ \tilde{\phi}^k }_{-1,h}^2
\\
\nonumber
&&+ (\frac{\varepsilon_1^2}{2}+\frac{\varepsilon_2^2}{2})\dt \sum_{k=1}^n\nrm{ \nabla_h \tilde{\phi}^{k+1} }_2^2
+ \frac{2 \dt}{\varepsilon^2}\sum_{k=1}^n\nrm{ \tau^k }_{-1,h}^2.
\label{CH_LOG-convergence-8}
     \end{eqnarray}
We observe $\nrm{ \mathbf{p}^{n+1} }_{-1,G}^2 \geq \frac{1}{2 }\nrm{ \tilde{\phi}^{n+1} }_{-1,h}^2$. Let $\frac{\varepsilon_1^2}{2}+\frac{\varepsilon_2^2}{2}=\frac{1}{18}$, we have
     \begin{equation}
      (\frac{1}{2}-\frac{\varepsilon^2\dt}{2}) \nrm{ \tilde{\phi}^{n+1} }_{-1,h}^2
      +\frac{\dt}{18}\sum_{k=1}^n\nrm{ \nabla_h \tilde{\phi}^{k+1} }_2^2
\leq \left(\frac{32\chi^2 \rho^2}{\varepsilon_1^2} + \frac{8\chi^2 \rho^2}{\varepsilon_2^2} + \frac{\varepsilon^2}{2}\right) \dt \sum_{k=1}^n \nrm{ \tilde{\phi}^k }_{-1,h}^2
+ \frac{2 \dt}{\varepsilon^2}\sum_{k=1}^n\nrm{ \tau^k }_{-1,h}^2.
     \end{equation}
By taking $\varepsilon^2 \dt < 1$, we get the following estimate by using the discrete Gronwall inequality
     \begin{equation}
\nrm{ \tilde{\phi}^{n+1} }_{-1,h} +\left(\frac{\dt}{18}\sum_{k=1}^n\nrm{ \nabla_h \tilde{\phi}^{k+1} }_2^2\right)^{\nicefrac{1}{2}} \le C ( \dt^2 + h^2 ) .
	\label{CH_LOG-convergence-9}
	 \end{equation}
This completes the proof.
\end{proof}

\section{Numerical results}    \label{sec:numerical results}
In this section, we use the proposed second-order BDF scheme \eqref{scheme-CH_LOG-1}-\eqref{scheme-mu-0} to numerically solve the MMC-TDGL model.
\subsection{Nonlinear multigrid solvers}   \label{subsec:multigrid}

We use the nonlinear multigrid method for solving the semi-implicit numerical scheme \eqref{scheme-CH_LOG-1}-\eqref{scheme-mu-0}. The fully discrete finite-difference scheme \eqref{scheme-CH_LOG-1}-\eqref{scheme-mu-0} is formulated as follows: Find $\phi_{i,j}^{n+1}$ and $\mu_{i,j}^{n+1}$ in ${\mathcal C}_{\rm per}$ such that
     \begin{eqnarray}	
3\phi_{i,j}^{n+1} - 2\dt \Delta_h \mu_{i,j}^{n+1} = 4\phi_{i,j}^n - \phi_{i,j}^{n-1},&& \nonumber
    \\
\mu_{i,j}^{n+1} - \kappa'(\phi_{i,j}^{n+1})\left(a_x((D_x \phi^{n+1})^2)
                +a_y((D_y \phi^{n+1})^2)\right)_{i,j}&&\nonumber\\
                +  2d_x(A_x\kappa(\phi^{n+1})D_x \phi^{n+1})_{i,j} + 2d_y(A_y\kappa(\phi^{n+1})D_y \phi^{n+1})_{i,j}&&\nonumber\\
                +  A \dt \Dh\phi_{i,j}^{n+1} - S'(\phi_{i,j}^{n+1})
                = H'(2\phi_{i,j}^n-\phi_{i,j}^{n-1}) + A \dt \Dh\phi_{i,j}^n.&&
                \nonumber
	 \end{eqnarray}
Denote
${\bf u}=(\phi_{i,j}^{n+1}, \mu_{i,j}^{n+1})^T$. Then the above discrete nonlinear system can be written in terms of a nonlinear operator ${\bf N}$ and the source term ${\bf S}$ such that
     \begin{eqnarray}\label{eqn:dis-nonlinear-system}
{\bf N}({\bf u})= {\bf S} .
     \end{eqnarray}
The $2\times N\times N$ nonlinear operator ${\bf N}({\bf u}^{n+1})=\big(N_{i,j}^{(1)}({\bf u}), N_{i,j}^{(2)}({\bf u})\big)^T$ can be defined as
     \begin{eqnarray*}
N_{i,j}^{(1)}({\bf u})&=& 3\phi_{i,j}^{n+1} - 2\dt \Delta_h \mu_{i,j}^{n+1},  \label{scheme-MMC-o1}
\\
N_{i,j}^{(2)}({\bf u})&=&\mu_{i,j}^{n+1} - \kappa'(\phi_{i,j}^{n+1})\left(a_x((D_x \phi^{n+1})^2)
                +a_y((D_y \phi^{n+1})^2)\right)_{i,j}\nn\\
                &&+  2d_x(A_x\kappa(\phi^{n+1})D_x \phi^{n+1})_{i,j} + 2d_y(A_y\kappa(\phi^{n+1})D_y \phi^{n+1})_{i,j}\nn\\
                &&+  A \dt \Dh\phi_{i,j}^{n+1} - S'(\phi_{i,j}^{n+1}),
\label{scheme-MMC-o2}
     \end{eqnarray*}
and the $2\times N\times N$ source ${\bf S} =\big(S_{i,j}^{(1)}, S_{i,j}^{(2)}\big)^T$ is given by
     \begin{eqnarray*}
S_{i,j}^{(1)}&=& 4\phi_{i,j}^n - \phi_{i,j}^{n-1} ,  \label{scheme-MMC-s1}
\\
S_{i,j}^{(2)} &=&H'(2\phi_{i,j}^n-\phi_{i,j}^{n-1}) + A \dt \Dh\phi_{i,j}^n.
\label{scheme-MMC-s2}
     \end{eqnarray*}
The system \eqref{eqn:dis-nonlinear-system} can be efficiently solved  using a nonlinear Full Approximation Scheme (FAS) multigrid method, as reported in earlier works
\cite{wise10,Chen2017A,guo16,baskaran13a,feng2016bsam,hu09}. Here we only provide the details of nonlinear smoothing scheme. For smoothing operator, we use a nonlinear Gauss-Seidel method with Red-Black ordering.

Let $k$ be the smoothing iteration.
Then the smoothing scheme is given by: for every $(i,j)$, stepping  lexicographically from $(1,1)$ to $(N,N)$, find $\phi_{i,j}^{n+1,k+1}, \mu_{i,j}^{n+1,k+1}$ that solve
     \begin{small}
     \begin{eqnarray}
3\phi_{i,j}^{n+1,k+1} +\frac{8\dt}{h^2}\mu_{i,j}^{n+1,k+1}\nonumber
=\tilde{S}_{i,j}^{(1)},&&\\
\mu_{i,j}^{n+1,k+1} - \frac{\kappa'(\phi_{i,j}^{n+1,k})}{\phi_{i,j}^{n+1,k}}\left(a_x((D_x \phi)^2)
                +a_y((D_y \phi)^2)\right)_{i,j}^{n+1,k}&&\nn\\
-\frac{1}{h^2}\left(\kappa(\phi_{i+1,j}^{n+1,k}) +\kappa(\phi_{i-1,j}^{n+1,k+1}) +\kappa(\phi_{i,j+1}^{n+1,k}) +\kappa(\phi_{i,j-1}^{n+1,k+1})+4 \kappa(\phi_{i,j}^{n+1,k}) \right)\phi_{i,j}^{n+1,k+1}&&\nn\\
+ \frac{\kappa(\phi_{i,j}^{n+1,k})}{ h^2\phi_{i,j}^{n+1,k}}(\phi_{i+1,j}^{n+1,k}+ \phi_{i-1,j}^{n+1,k+1} + \phi_{i,j+1}^{n+1,k} + \phi_{i,j-1}^{n+1,k+1})\phi_{i,j}^{n+1,k+1}&&\nn\\
-\left(\frac{4A\dt}{h^2} + S''(\phi_{i,j}^{n+1,k})\right)\phi_{i,j}^{n+1,k+1}
= \tilde{S}_{i,j}^{(2)},&&\nonumber
     \end{eqnarray}
     \end{small}
where
     \begin{equation*}
     \begin{aligned}
\tilde{S}_{i,j}^{(1)}:=& S_{i,j}^{(1)}
+\frac{2\dt}{h^2}( \mu_{i+1,j}^{n+1,k}+\mu_{i-1,j}^{n+1,k+1}+\mu_{i,j+1}^{n+1,k}+\mu_{i,j-1}^{n+1,k+1}),\nn\\
\tilde{S}_{i,j}^{(2)}:=& S_{i,j}^{(2)}- S''(\phi_{i,j}^{n+1,k})\phi_{i,j}^{n+1,k} + S'(\phi_{i,j}^{n+1,k})\nn\\
&-\frac{1}{h^2}
\Bigl(\kappa(\phi_{i+1,j}^{n+1,k})\phi_{i+1,j}^{n+1,k} +\kappa(\phi_{i-1,j}^{n+1,k+1})\phi_{i-1,j}^{n+1,k+1}\nn\\ &\quad\quad+\kappa(\phi_{i,j+1}^{n+1,k})\phi_{i,j+1}^{n+1,k} +\kappa(\phi_{i,j-1}^{n+1,k+1})\phi_{i,j-1}^{n+1,k+1}\Bigr)\nn\\
& -\frac{A\dt}{h^2}(\phi_{i+1,j}^{n+1,k}+ \phi_{i-1,j}^{n+1,k+1} + \phi_{i,j+1}^{n+1,k} + \phi_{i,j-1}^{n+1,k+1}).
     \end{aligned}
     \end{equation*}
The above linearized system, which comes from  a local Newton approximation of the logarithmic term and a local linearization of other nonlinear terms in the Gauss-Seidel scheme, can be solved by the Cramer's Rule.
\subsection{Numerical experiments}    \label{subsec:numerical results}
In this part, we perform some numerical simulations for the scheme \eqref{scheme-CH_LOG-1}-\eqref{scheme-mu-0} to verify the theoretical results including energy decay, mass conservation, the second order accuracy and positivity of the numerical  solution. For this, we will present three numerical examples with different initial conditions.

We use the domain $\Omega = [0,64]^2$ in 2D, $\Omega = [0,64]^3$ in 3D and choose the parameters in the model as $\chi =2.37, N_2 = 0.16, N_1 = 5.12$. In addition, we set $A = \chi^2\rho^2$.
\begin{example}\label{example 1}
The initial data is chosen as
\begin{eqnarray}
\phi_0(x,y) = 0.6+0.15\cos\big({3\pi x}/{32}\big)\cos\big({3\pi y}/{32}\big),\label{eqn:init1}
\end{eqnarray}
\end{example}
and this problem is subject to periodic boundary condition. The time step is $\dt = 0.001$.
This example is designed to study the numerical accuracy in time and space.
\begin{figure}[ht]
	\begin{center}
		\begin{subfigure}{}
			\includegraphics[width=1.9in]{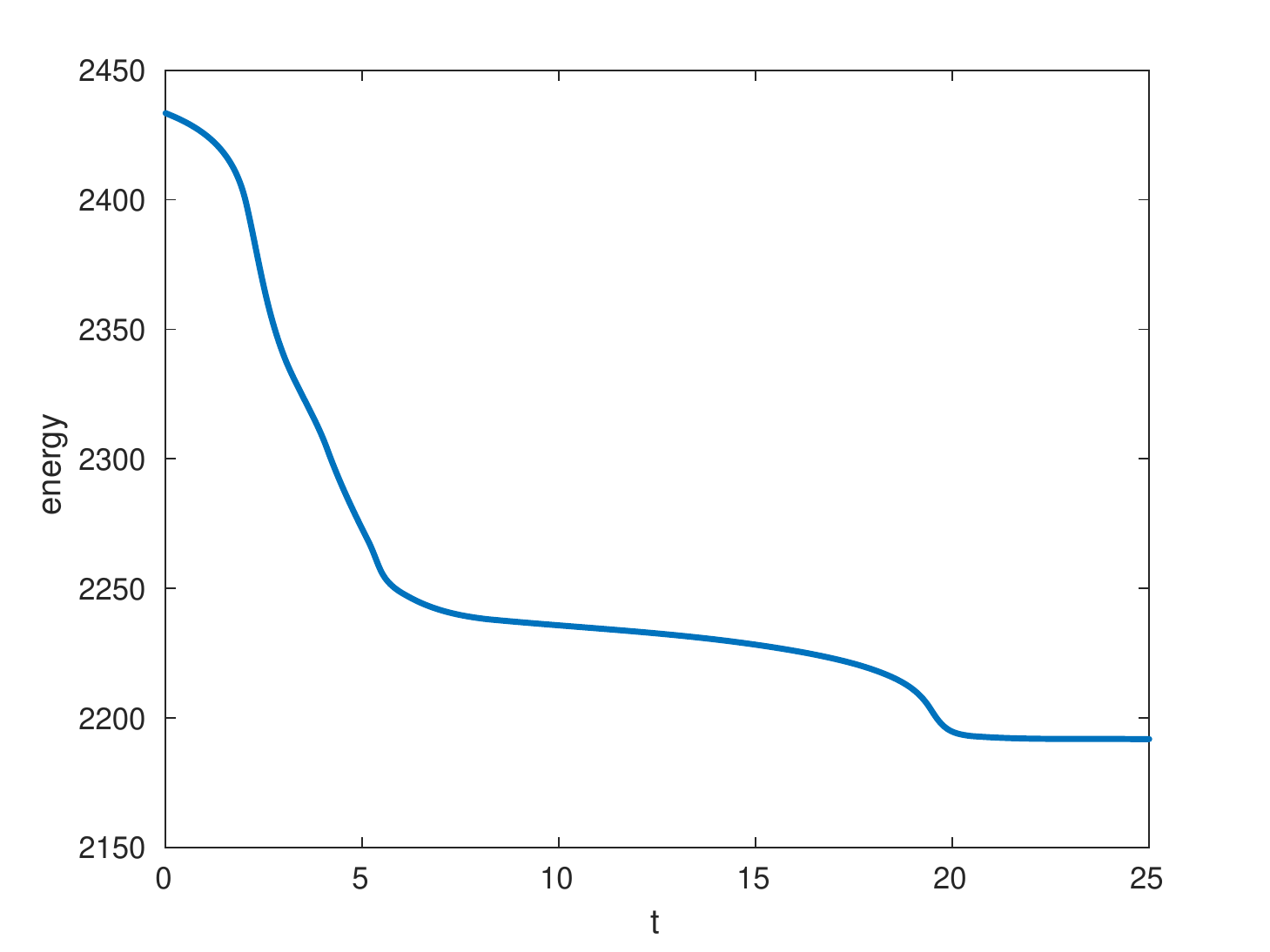}
		\end{subfigure}
		\begin{subfigure}{}
			\includegraphics[width=1.9in]{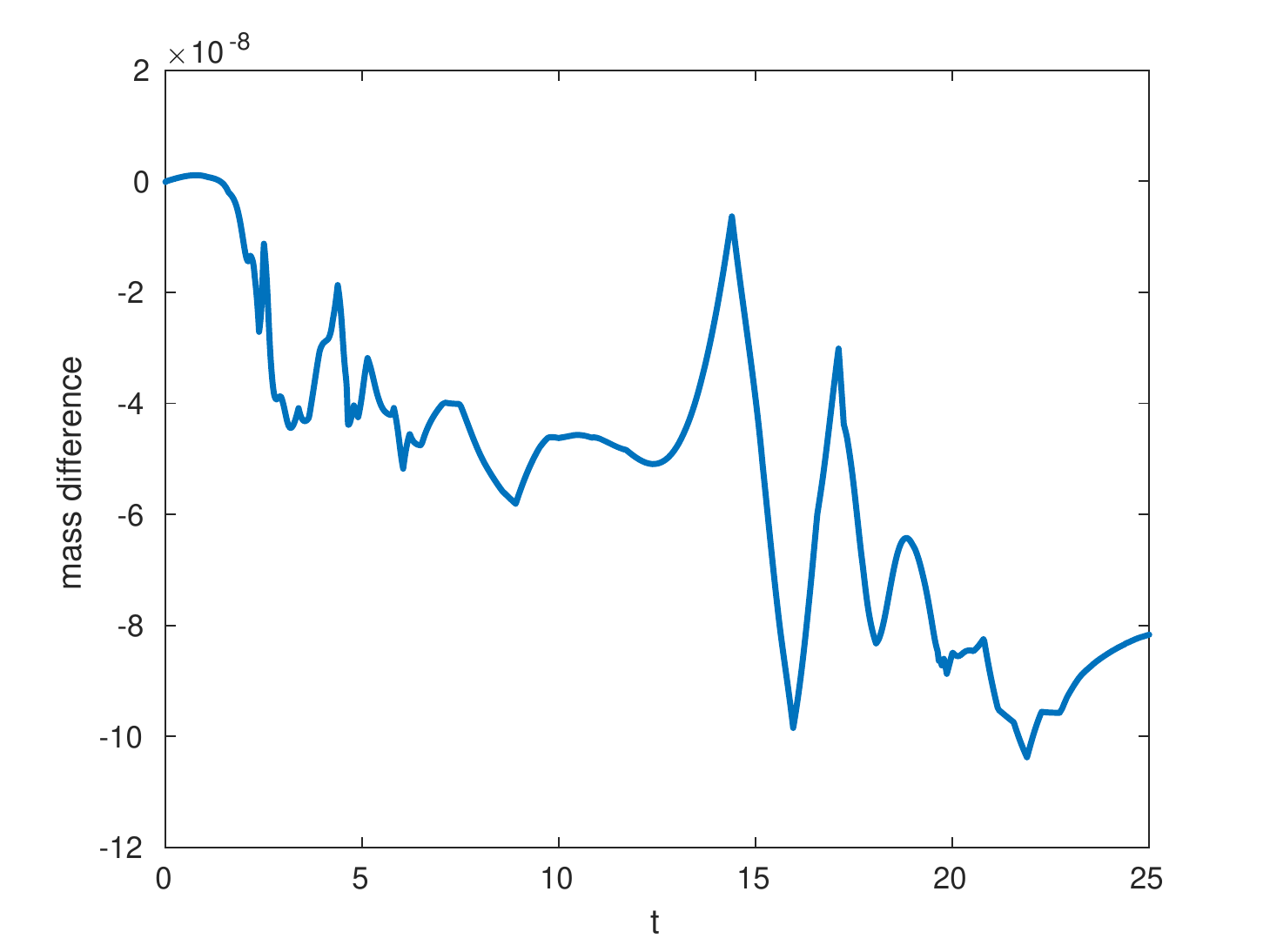}
		\end{subfigure}
	\end{center}
	\caption{Example \ref{example 1}: the left is the energy evolution with time and the right is the error development of the total mass. }
	\label{fig:cosenergymass}
\end{figure}
In the left of the Fig~\ref{fig:cosenergymass}, it illustrates the energy evolution, which indicates energy decay with time. \\
We present the evolution of the mass difference of $\phi$ computed as $\overline{\phi^n}-\overline{\phi}_0$, where $\overline{\phi^n}$ is defined in~\eqref{mass_define}.
The rough estimate of the difference of the total mass of $\phi$ is presented in the right of the Fig~\ref{fig:cosenergymass}, which means that the property stated in \eqref{mass_conserv-1} is verified numerically.

\begin{figure}[!htp]
\begin{center}	
\includegraphics[width=1.9in]{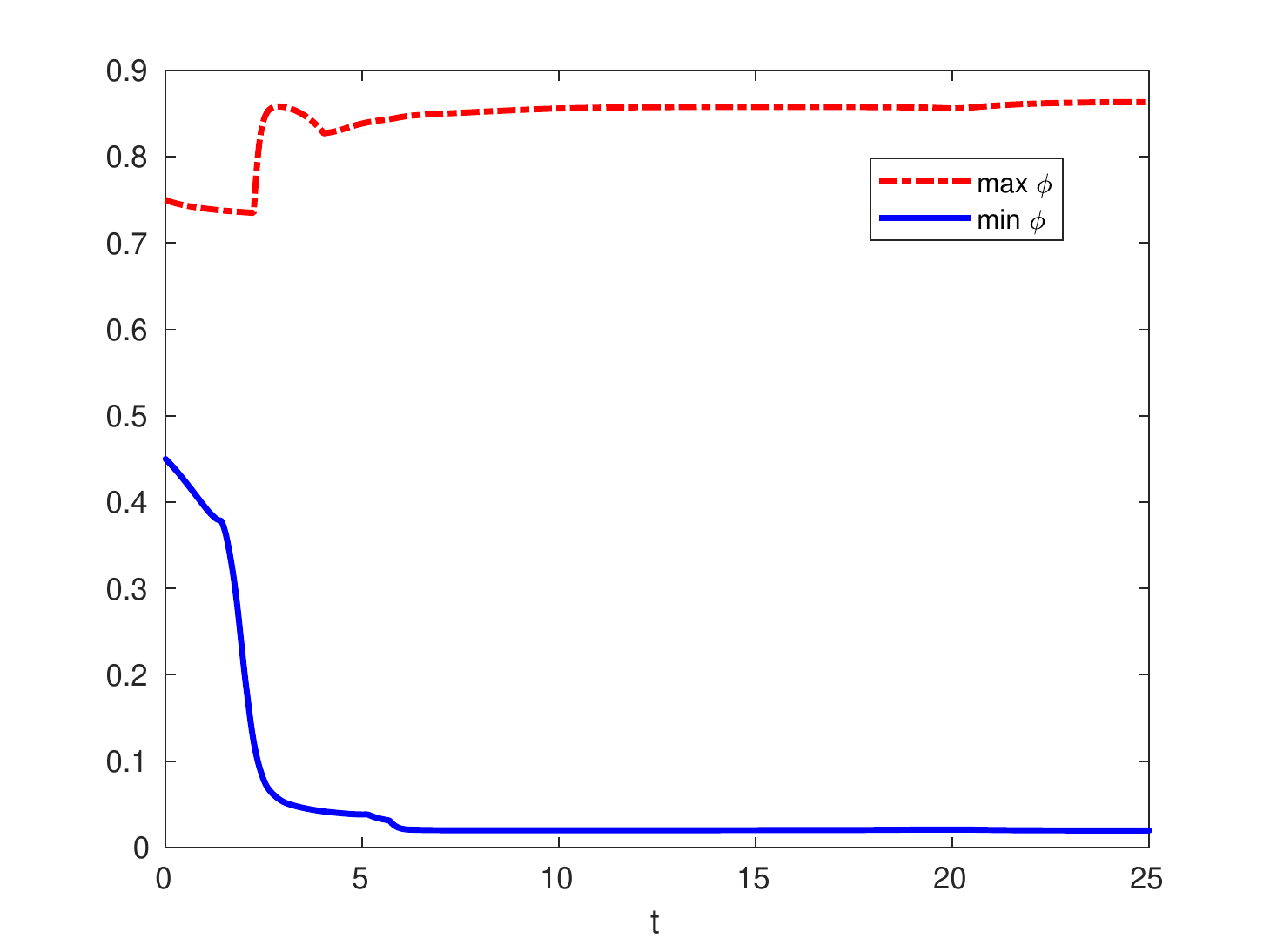}
	\caption{Example \ref{example 1}: the maximum and minimum values with time. }\label{fig:cosmaxmin}
\end{center}
\end{figure}

In Fig~\ref{fig:cosmaxmin}, we plot the maximum and minimum values of $\phi_{i,j}^n$ with time developing. It is observed that the numerical solution well remains in the interval of $(0,\nicefrac{1}{\rho})$.

In order to test the second order convergence, we use a linear refinement path, \emph{i.e.}, $\dt=Ch, C = 0.0002$. At the final time $T=0.128$, we expect the global error to be $\mathcal{O}(\dt^2)+\mathcal{O}(h^2)=\mathcal{O}(h^2)$ under the  $\ell^2$ norm, as $h, \dt \to 0$.  Since we do not have an exact solution, instead of calculating the error at the final time, we compute the Cauchy difference, which is defined as $\delta_\phi: =\phi_{h_f}-\mathcal{I}_c^f(\phi_{h_c})$, where $\mathcal{I}_c^f$ is a bilinear interpolation operator (We applied Nearest Neighbor Interpolation in Matlab, see \cite{Chen2017A,feng2016fch,feng2016preconditioned}). This requires having a relatively coarse solution, parametrized by $h_c$, and a relatively fine solution, parametrized by $h_f$, where $h_c = 2 h_f$, at the same final time. The $\ell^2$ norms of Cauchy difference and the convergence rates can be found in Table~\ref{tab:cov}. The  results confirm our expectation for the convergence order.

\begin{table}[!htb]
\begin{center}
\caption{Errors and convergence rates. Parameters are given in the
		text, and the initial data are defined in \eqref{eqn:init1}. The refinement path is $\dt=0.0002h$.} \label{tab:cov}
\begin{tabular}{cccccc}
\hline Grid sizes&$16^2$ &$32^2$& $64^2$& $128^2$& $256^2$  \\
\hline Error& $ 4.0436\text{E-}01$&$ 1.0328\text{E-}02$ &$ 2.5941\text{E-}02$&$ 6.4847\text{E-}03$&$ 1.6171\text{E-}03$
\\Rate&-& 1.97  &   1.99&   2.00&   2.00
\\
\hline
\end{tabular}
\end{center}
\end{table}
\begin{figure}[ht]
	\begin{center}
		\begin{subfigure}{}
			\includegraphics[width=4.5in]{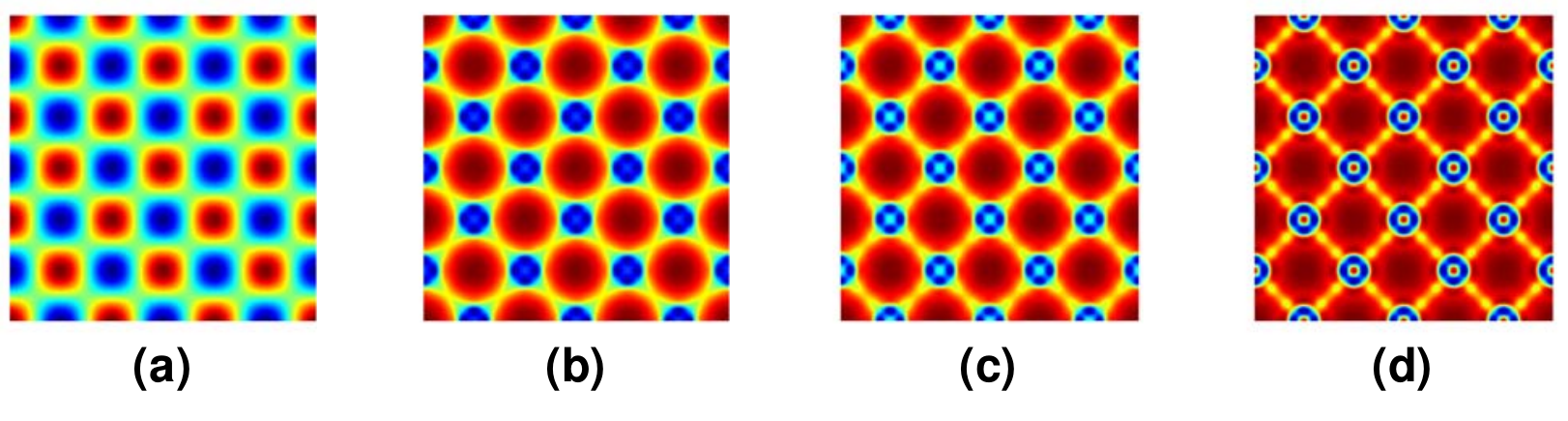}
		\end{subfigure}
        \vskip -0.3cm
		\begin{subfigure}{}
			\includegraphics[width=4.5in]{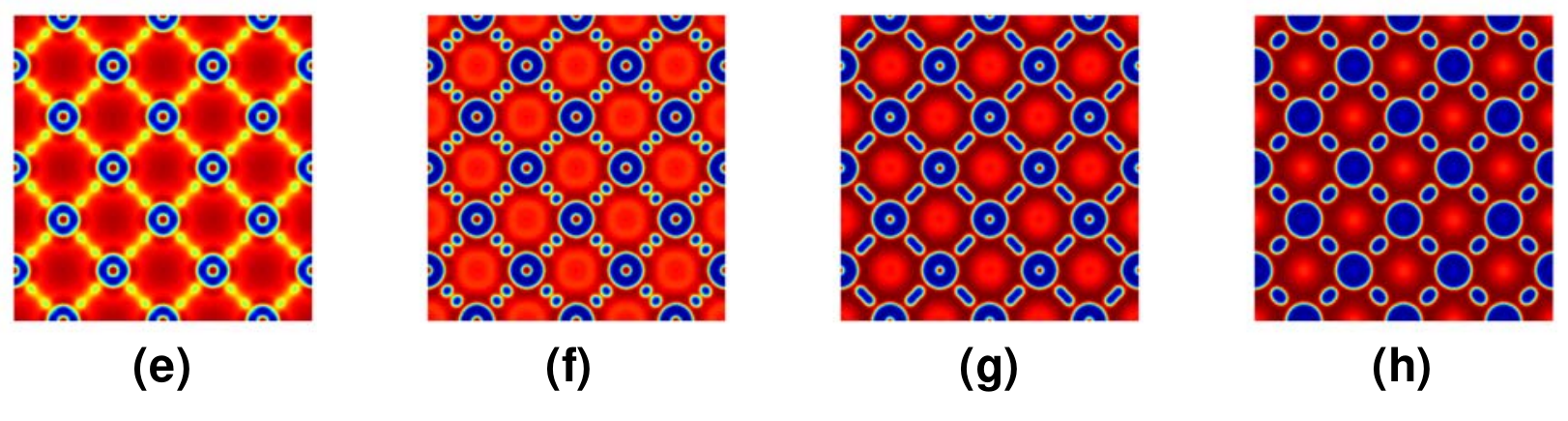}
		\end{subfigure}
        \vskip -0.3cm
		\begin{subfigure}{}
			\includegraphics[width=4.5in]{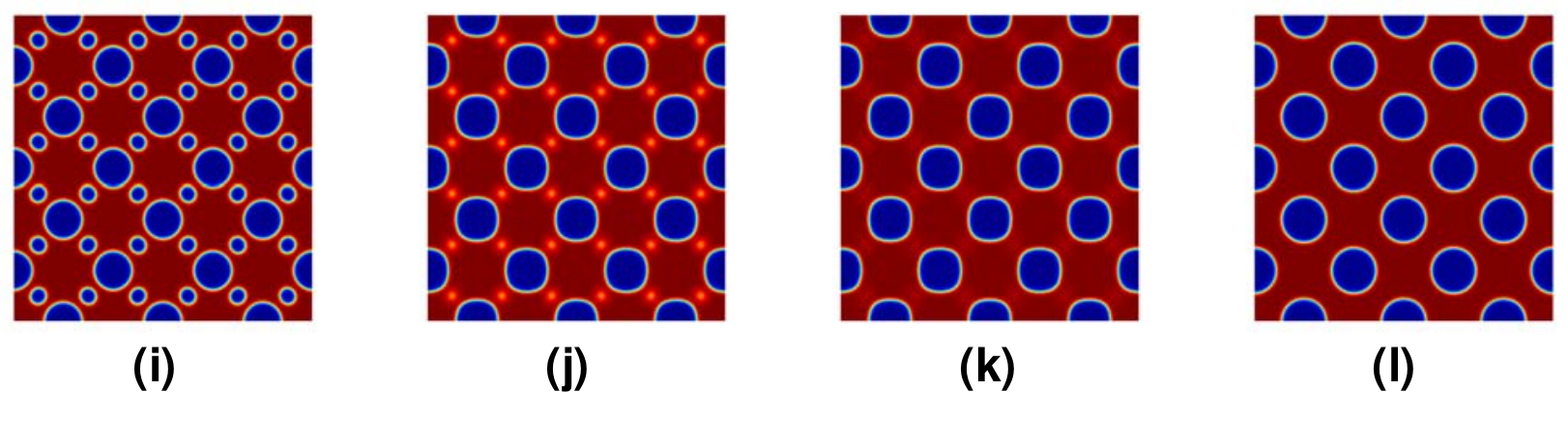}
		\end{subfigure}
	\end{center}
    \vskip -0.7cm
	\caption{Example \ref{example 1}: the phase evolution of $\phi$. This computation is done using the BDF2 $A = \chi^2\rho^2$ scheme. $(a-l)$ are corresponding to
$t = 0, 1.7, 1.9, 2.2, 2.3, 3.5, 4.5, 5.5, 9.5, 19.5, 20~ \text{and}~ 25$.
$ \dt= 1.0 \times 10^{-3}, N = 256.$ }
\label{fig:long-time-cos}
\end{figure}

Fig~\ref{fig:long-time-cos} describes the evolution of $\phi$ at some selected time levels with the
initial condition \eqref{eqn:init1}. The numerical results are consistent with the experiments on
this topic in \cite{Zhai2012Investigation}.

We present the error comparison of four schemes for the MMC-TDGL equation: classical first-order
convex splitting scheme(CS1), the scheme \eqref{scheme-CH_LOG-1}-\eqref{scheme-CH_initial} with
$A=\chi^2\rho^2$, the scheme \eqref{scheme-CH_LOG-1}-\eqref{scheme-CH_initial} without
regularization term, i.e., $A=0$, and the standard BDF2(full implicit) scheme in the tables
~\ref{tab:cov_compare_1} and ~\ref{tab:cov_compare_2}. The standard BDF2(full implicit) scheme
shows excellent accuracy in the finite short time. BDF2 $A=0$ scheme has more extra error than the standard BDF2(full implicit) scheme, owing to the explicit expression of the concave
term. The
scheme \eqref{scheme-CH_LOG-1}-\eqref{scheme-CH_initial} with $A=\chi^2\rho^2$  has more extra error
than the BDF2 with $A=0$ due to the existence of the regularization term.The standard BDF2(full implicit) scheme is convergent in the early stage, but it is not convergent in the late stage.

Fig~\ref{fig:energy compare} shows the energy comparison of three schemes for the MMC-TDGL equation:
classical first-order convex splitting scheme(CS1), the scheme \eqref{scheme-CH_LOG-1}-\eqref{scheme-CH_initial} with
$A=\chi^2\rho^2$ and the scheme \eqref{scheme-CH_LOG-1}-\eqref{scheme-CH_initial} without
regularization term, i.e., $A=0$. The three energy plots are all non-increasing with time. The development of energy using the scheme \eqref{scheme-CH_LOG-1}-\eqref{scheme-CH_initial} with
$A=\chi^2\rho^2$ is almost same to that using the scheme \eqref{scheme-CH_LOG-1}-\eqref{scheme-CH_initial} without
regularization term, i.e., $A=0$. There exists an accepted energy error using  CS1 scheme.

\begin{table}[!htb]
\begin{center}
\caption{$\dt = 1\times 10^{-3}$, $T=1.6$.} \label{tab:cov_compare_1}
\begin{tabular}{ccccccccc}
\hline
&Scheme& Maxerr & L2err & CPU \\
\hline &CS1& 2.0085e-03 & 2.1857e-02&0.9907

\\ & BDF2 A=0 & 1.1982e-04 &1.5000e-03& 0.9117
\\ & BDF2 $A=\chi^2\rho^2$& 1.3426e-04 & 1.5637e-03 &1.0343
\\ & Standard BDF2 &1.0538e-04  & 1.3304e-03 &1.1529
\\
\hline
\end{tabular}
\end{center}
\end{table}
\begin{table}[!htb]
\begin{center}
\caption{$\dt = 2\times 10^{-3}$, $T=1.6$.}\label{tab:cov_compare_2}
\begin{tabular}{ccccccccc}
\hline
&Scheme& Maxerr & L2err & CPU \\
\hline &CS1& 4.0044e-03 & 4.3772e-02&1.3070
\\ & BDF2 A=0 & 3.0604e-04&3.6000e-03&1.1111
\\ & BDF2 $A=\chi^2\rho^2$& 3.6588e-04& 4.1274e-03&1.2122
\\ & Standard BDF2 &2.6889e-04& 3.4268e-03& 1.8277
\\
\hline
\end{tabular}
\end{center}
\end{table}

\begin{figure}[ht]
	\begin{center}
			\includegraphics[width=4.3in]{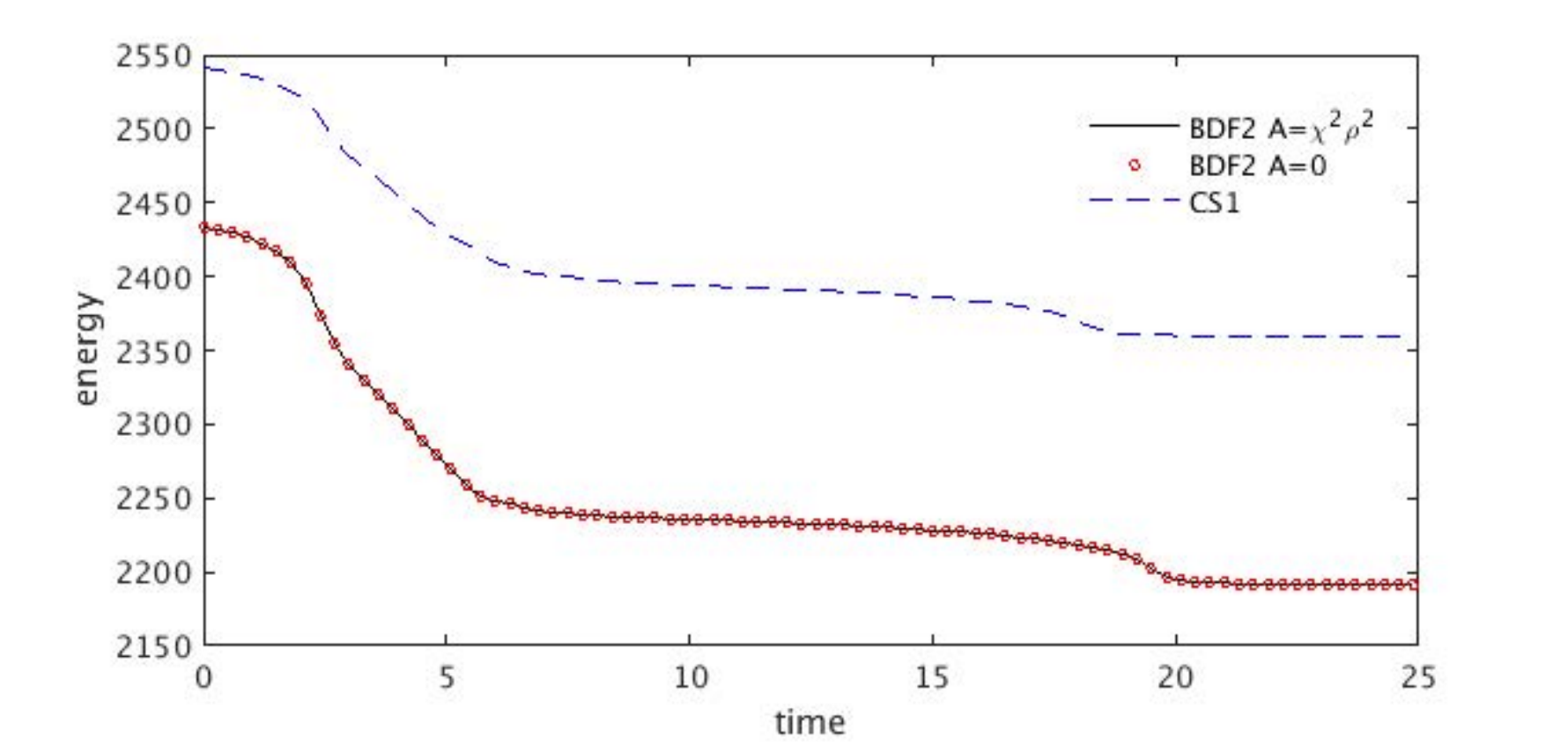}
	\end{center}
	\caption{Example \ref{example 2}: The energy evolution with time using different schemes
  classical first-order convex splitting scheme(CS1), the scheme \eqref{scheme-CH_LOG-1}-\eqref{scheme-CH_initial} with $A=\chi^2\rho^2$, and the scheme \eqref{scheme-CH_LOG-1}-\eqref{scheme-CH_initial} with $A=0$. }
	\label{fig:energy compare}
\end{figure}

\begin{example}\label{example 2}
The initial data is chosen as:
\begin{eqnarray}\label{eqn:init2}
\phi_0(x,y) = 0.6+r_{i,j},
\end{eqnarray}
where the $r_{i,j}$ are uniformly distributed random numbers in [-0.15, 0.15].
\end{example}

\begin{figure}[ht]
	\begin{center}
		\begin{subfigure}{}
			\includegraphics[width=1.9in]{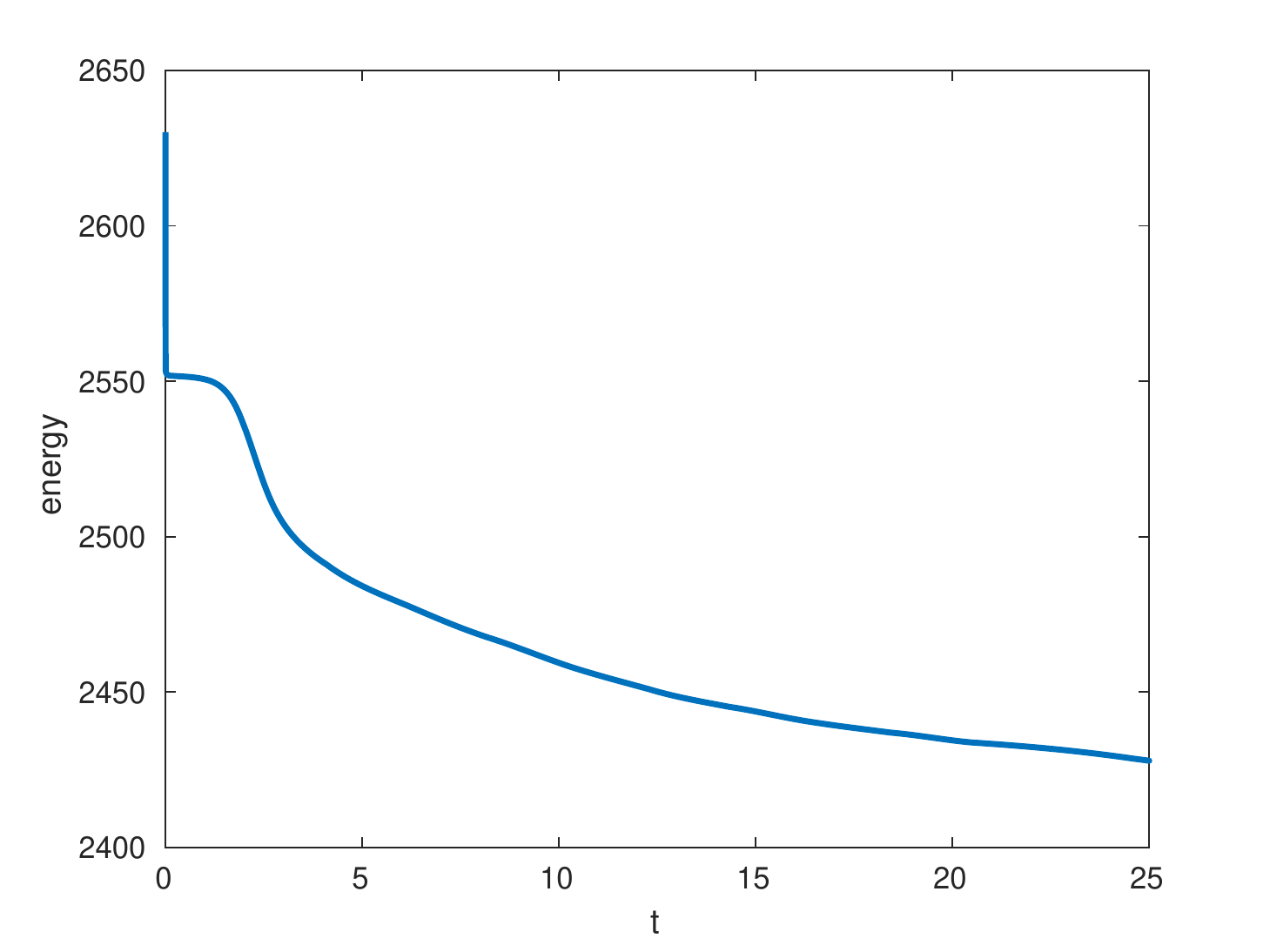}
		\end{subfigure}
		\begin{subfigure}{}
			\includegraphics[width=1.9in]{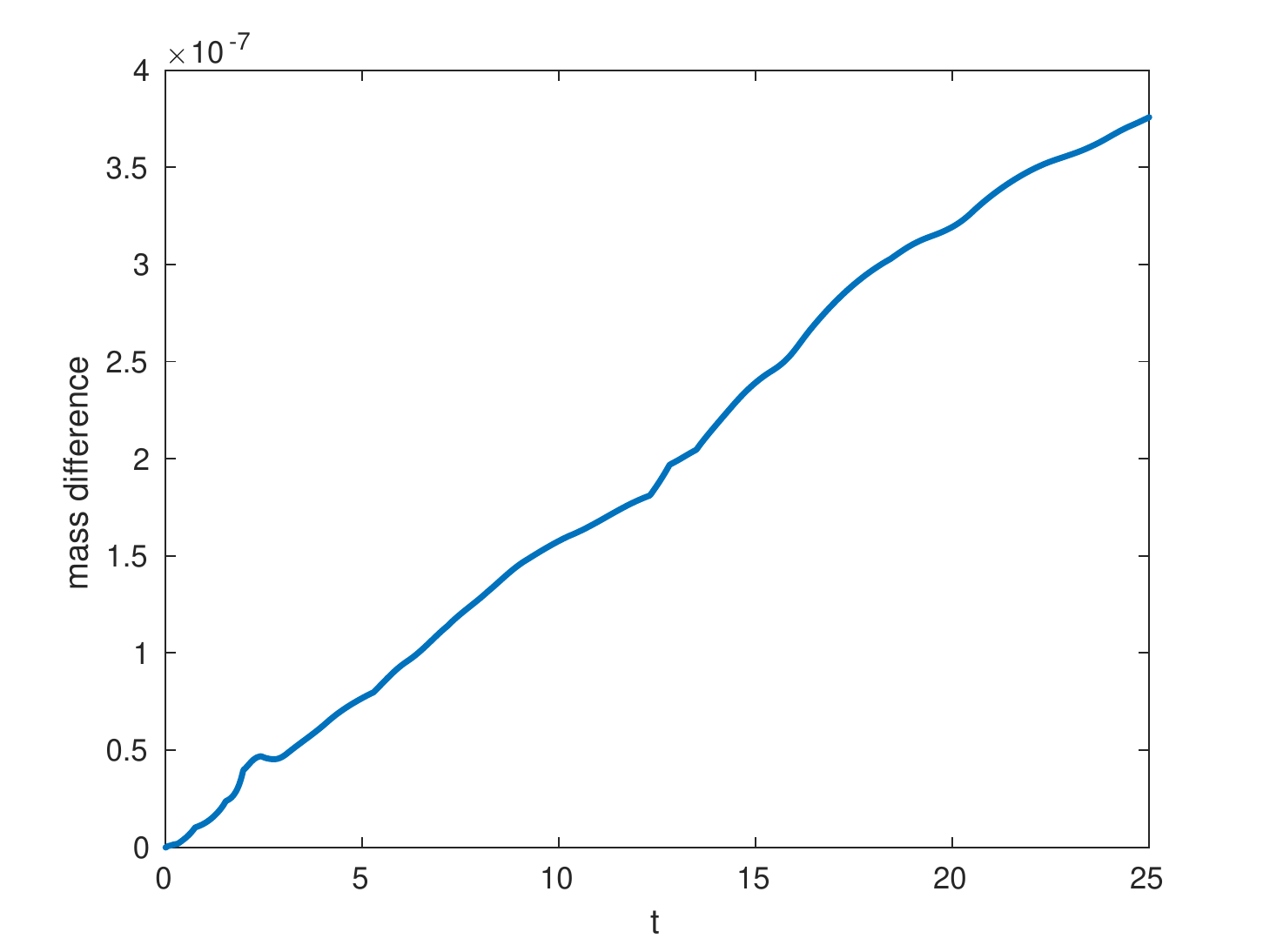}
		\end{subfigure}
	\end{center}
	\caption{Example \ref{example 2}: the left is the energy evolution with time and the right is the error development of the total mass. }
	\label{fig:randomenergymass}
\end{figure}

In the left of the Fig~\ref{fig:randomenergymass}, we show the energy evolution, which proves the energy decay with time. In the computation, the mass conservation of $\phi$ is also numerically observed from the right of the Fig~\ref{fig:randomenergymass}, which  is similar to the one shown in Fig ~\ref{fig:cosenergymass}.\\
\begin{figure}[!htp]
\begin{center}	
\includegraphics[width=1.9in]{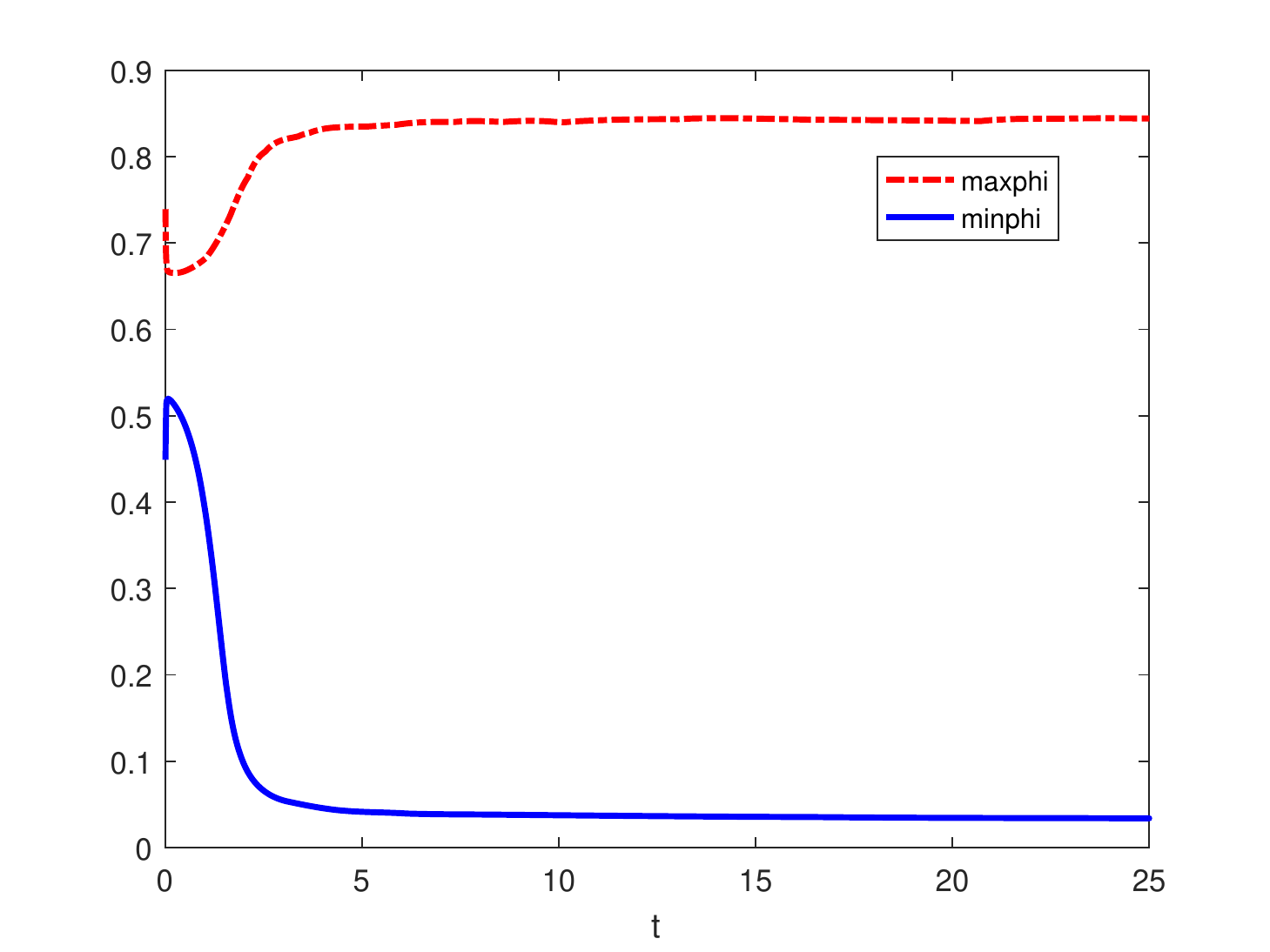}
	\caption{Example \ref{example 2}: the maximum and minimum values with time. }\label{fig:randommaxmin}
\end{center}
\end{figure}

In Fig~\ref{fig:randommaxmin}, we present the maximum and minimum value of the numerical solution with time with random initial value. It's seen that the positivity property is examined numerically.
\begin{figure}[ht]
	\begin{center}
		\begin{subfigure}{}
			\includegraphics[width=4.5in]{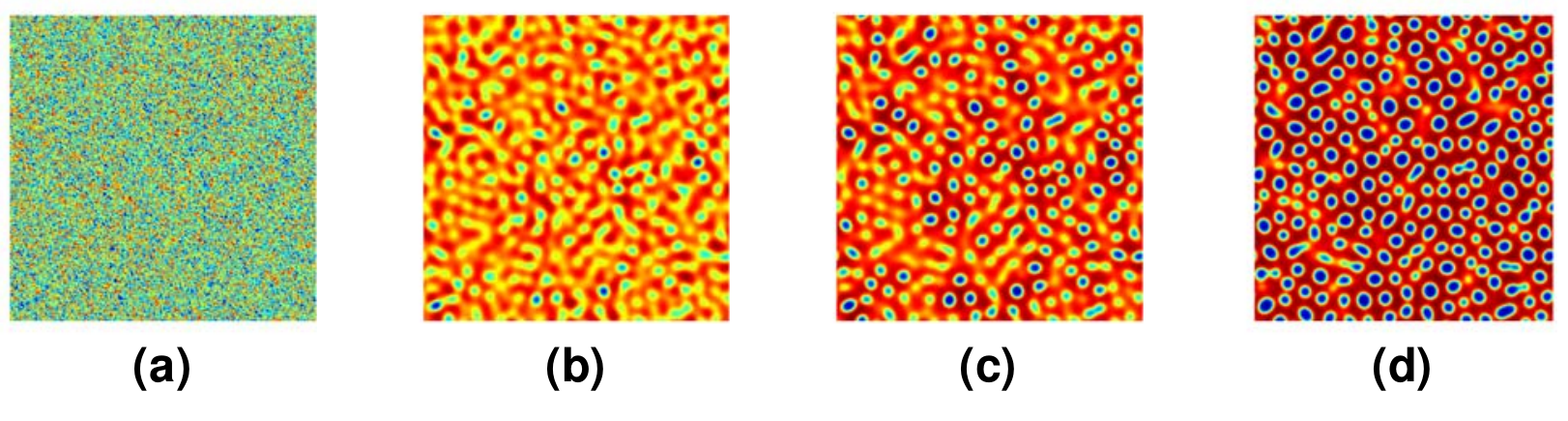}
		\end{subfigure}
        \vskip -0.3cm
		\begin{subfigure}{}
			\includegraphics[width=4.5in]{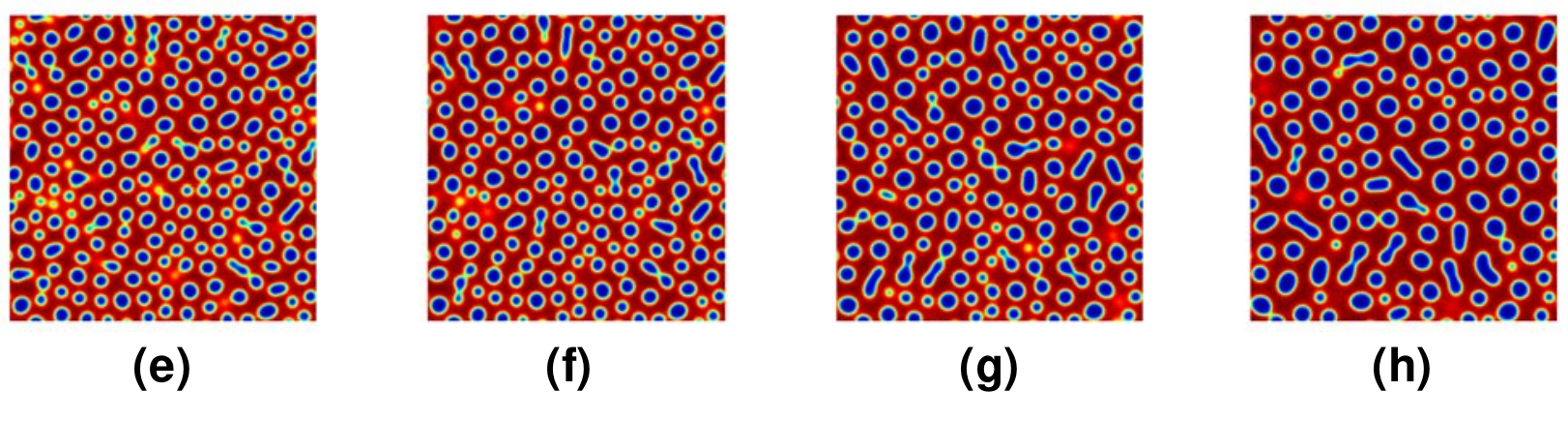}
		\end{subfigure}
        \vskip -0.3cm
		\begin{subfigure}{}
			\includegraphics[width=4.5in]{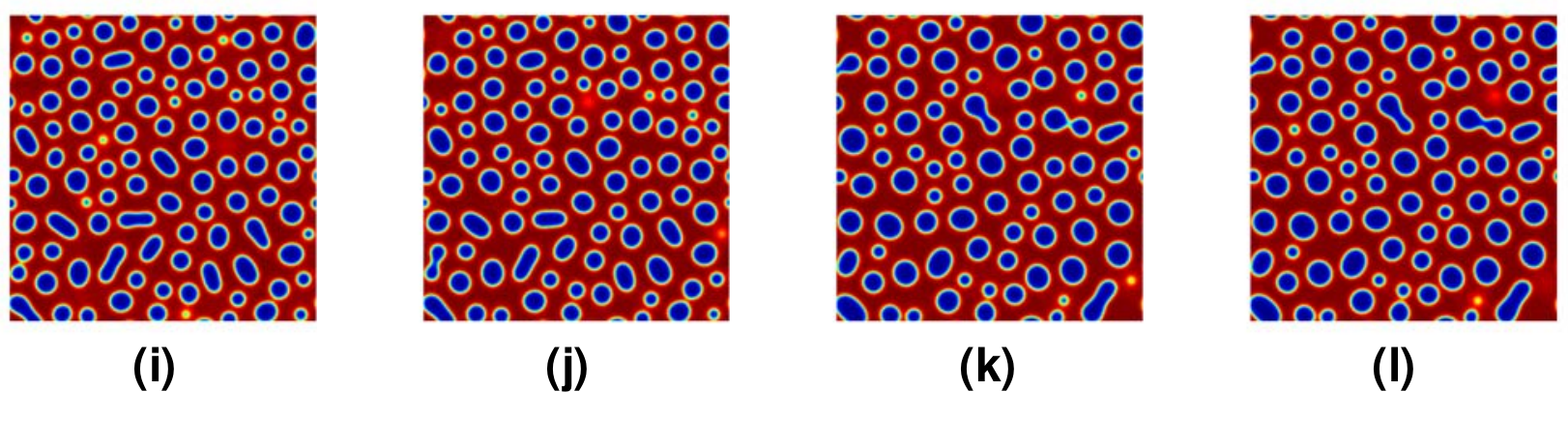}
		\end{subfigure}
	\end{center}
	\caption{Example \ref{example 2}: $t = 0, 1, 2, 3, 4, 5, 7, 11, 15, 17, 24~ \text{and}~ 25$. $\dt= 1.0 \times 10^{-3}, N = 256.$ }
	\label{fig:long_time_random}
\end{figure}

In Fig~\ref{fig:long_time_random}, we present the evolution of $\phi$ at different time with the initial data \eqref{eqn:init2}. In order to compare with the results obtained by Li in \cite{Li2016An}, we choose the same parameters  and initial data in the model. The numerical results are similar to the ones shown in \cite{Li2016An}.

\begin{example}\label{example 3}
The initial data is chosen as
\begin{eqnarray}
\phi_0(x,y,z) = 0.6+0.15\cos\big({3\pi x}/{32}\big)\cos\big({3\pi y}/{32}\big)\cos\big({3\pi z}/{32}\big),\label{eqn:init3}
\end{eqnarray}
\end{example}
\begin{figure}[ht]
	\begin{center}
		\begin{subfigure}{}
			\includegraphics[width=4.5in]{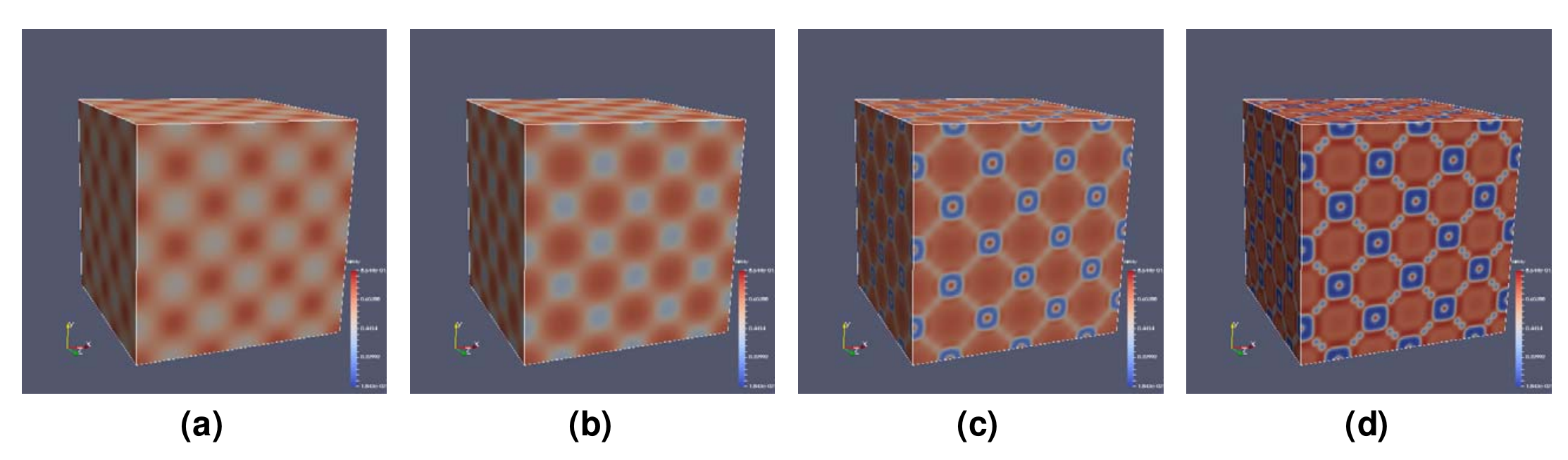}
		\end{subfigure}
        \vskip -0.3cm
		\begin{subfigure}{}
			\includegraphics[width=4.5in]{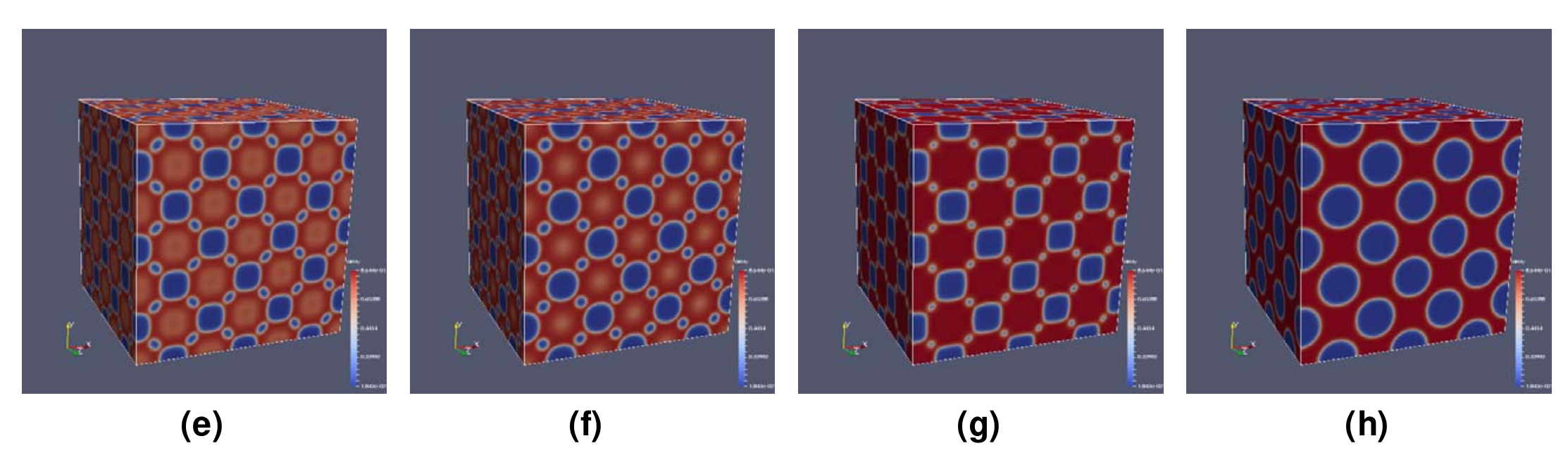}
		\end{subfigure}
	\end{center}
	\caption{Example \ref{example 3}: Three dimensional simulation. The contour plots of $\phi$, $(a-h)$ are corresponding to $t = 0, 1, 2, 3, 4, 5, 10~ \text{and}~ 25$. $ \dt= 1.0 \times 10^{-3}, N = 256.$ }
\label{fig:long-time-cos-3d}
\end{figure}
Fig~\ref{fig:long-time-cos-3d} describes the evolution of $\phi$ at some selected time levels with the initial condition \eqref{eqn:init3}. The numerical results are consistent with the Fig~\ref{fig:long-time-cos}.
\section{Conclusions} \label{sec:conclusion}
 Now we have presented a second order BDF scheme based on the convex splitting technique of the given energy functional for the MMC-TDGL equation, with a centered finite difference in space. A unique solvability and unconditional energy stability turn to be available. Moreover, the positivity-preserving property and the second order convergence analysis are available in the theoretical level. In addition, mass conservation, energy stability, bound of the numerical solution and the second order accurate are demonstrated in the numerical experiments. At last, we can see the details of the phase transition of the Macromolecular Microsphere Composite hydrogel.
\section*{Acknowledgments}
H.~Zhang is partly supported by the National Natural Science Foundation of China~(NSFC) Nos.11471046,11971002. Z.R.~Zhang is partly supported by the National Natural Science Foundation of China~(NSFC) No.11871105,11571045 and Science Challenge Project No. TZ2018002.

\bibliographystyle{plain}
\bibliography{MMCBDF_CICP}

\end{document}